\documentclass[10pt,a4paper,english]{article}

\usepackage{hyperref}
\usepackage{setspace}
\usepackage{lipsum}
\usepackage{authblk}
\usepackage[style=alphabetic,sorting=nyt,backend=bibtex,firstinits=true]{biblatex}
\AtBeginBibliography{\setstretch{1}\footnotesize}
\AtEveryBibitem{%
  \iffieldundef{url}{}{\clearfield{doi}}%
}

\usepackage{dsfont}
\usepackage{amsfonts,amssymb,stmaryrd,amsmath,amsthm,amssymb}
\usepackage{graphicx,color}
\usepackage{multimedia}
\usepackage{float}
\usepackage{placeins}

\usepackage{mathrsfs}

\usepackage{algorithm}

\usepackage{mathtools}

\usepackage[a4paper]{geometry}
\newtheorem{theorem}{Theorem}[section]
\newtheorem{lemma}{Lemma}[section]
\newtheorem{remark}{Remark}[section]
\newtheorem{proposition}{Proposition}[section]
\newtheorem{corollary}{Corollary}[section]

\newcommand \udotini {\dot{u}_0}

\newcommand \p {\partial}

\newcommand \R {\mathbb{R}}

\renewcommand \L {\mathrm{L}}
\newcommand \W {\mathrm{W}}
\newcommand \WW {\mathbf{W}}
\newcommand \WWW {\mathbb{W}}
\newcommand \B {\mathrm{B}}
\newcommand \BB {\mathbf{B}}

\newcommand \LL {\mathbf{L}}
\newcommand \LLL {\mathbb{L}}

\renewcommand \H {\mathrm{H}}
\newcommand \HH {\mathbf{H}}

\newcommand \I {\mathrm{I}}

\newcommand \III {\mathbb{I}}
\newcommand \Id {\mathrm{Id}}

\renewcommand \d {\mathrm{d}}

\renewcommand \det {\mathrm{det}}
\newcommand \trace {\mathrm{tr}}

\newcommand \cof {\mathrm{cof}}

\DeclareMathOperator{\divg}{div}

\begingroup\makeatletter\ifx\SetFigFont\undefined%
\gdef\SetFigFont#1#2#3#4#5{%
  \reset@font\fontsize{#1}{#2pt}%
  \fontfamily{#3}\fontseries{#4}\fontshape{#5}%
  \selectfont}%
\fi\endgroup%

\title{A damped elastodynamics system under the global injectivity condition: Local wellposedness in $L^p$-spaces}

\author[1, 2]{S\'ebastien Court}
\affil[1]{\begin{small}Department of Mathematics, University of Innsbruck, Technikerstrasse 13, 6020 Innsbruck, Austria.\end{small}}
\affil[2]{\begin{small}Digital Science Center, University of Innsbruck, Innrain 15, 6020 Innsbruck, Austria. Email: {\tt sebastien.court@uibk.ac.at}\end{small}}

\begin{document}

\maketitle

\begin{abstract}
The purpose of this paper is to model mathematically mechanical aspects of cardiac tissues. The latter constitute an elastic domain whose total volume remains constant. The time deformation of the heart tissue is modeled with the elastodynamics equations dealing with the displacement field as main unknown. These equations are coupled with a pressure whose variations characterize the heart beat. This pressure variable corresponds to a Lagrange multiplier associated with the so-called global injectivity condition. We derive the corresponding coupled system with nonhomogeneous boundary conditions where the pressure variable appears. For mathematical convenience a damping term is added, and for a given class of strain energies we prove the existence of local-in-time solutions in the context of the $L^p$-parabolic maximal regularity.
\end{abstract}

\hfill  

\newpage
\tableofcontents

\section{Introduction}
Heart beats enable the oxygenation of cardiac tissues and so guarantees an healthy electric activity. Mechanically, the oxygenation is related to the variations of the {\it hydraulic pressure}~\cite{Chernysh}, which is mathematically a Lagrange multiplier corresponding to the constraint of constant global volume of the heart tissues, as the latter are crossed by blood, considered to be an incompressible fluid. The time variation of this pressure quantifies shape variations of the heart via its deformation, and so determines the intensity of this contraction. In particular, this quantity is central when studying defibrillation, which is the underlying motivation of our contribution.

The focus of the present article lies in the mathematical modeling of the time deformations of the heart tissues, and related wellposedness questions for the derived system of partial differential equations dealing with the displacement field as main unknown. The heart is modeled as an hyperelastic tissue, crossed by blood, assumed to be an incompressible fluid, and so the total volume inside the heart has to remain constant through the time. Considering that the exterior part of the boundary of the domain is only subject to rigid displacements (see Figure~\ref{fig0}), this means that the total volume of the heart itself has to remain constant through the time. Therefore, beside the displacement field, we also introduce a pressure variable, that is a Lagrange multiplier associated with this constraint.\\
\begin{center}
\scalebox{0.4}{
	\begin{picture}(0,0)
	\includegraphics{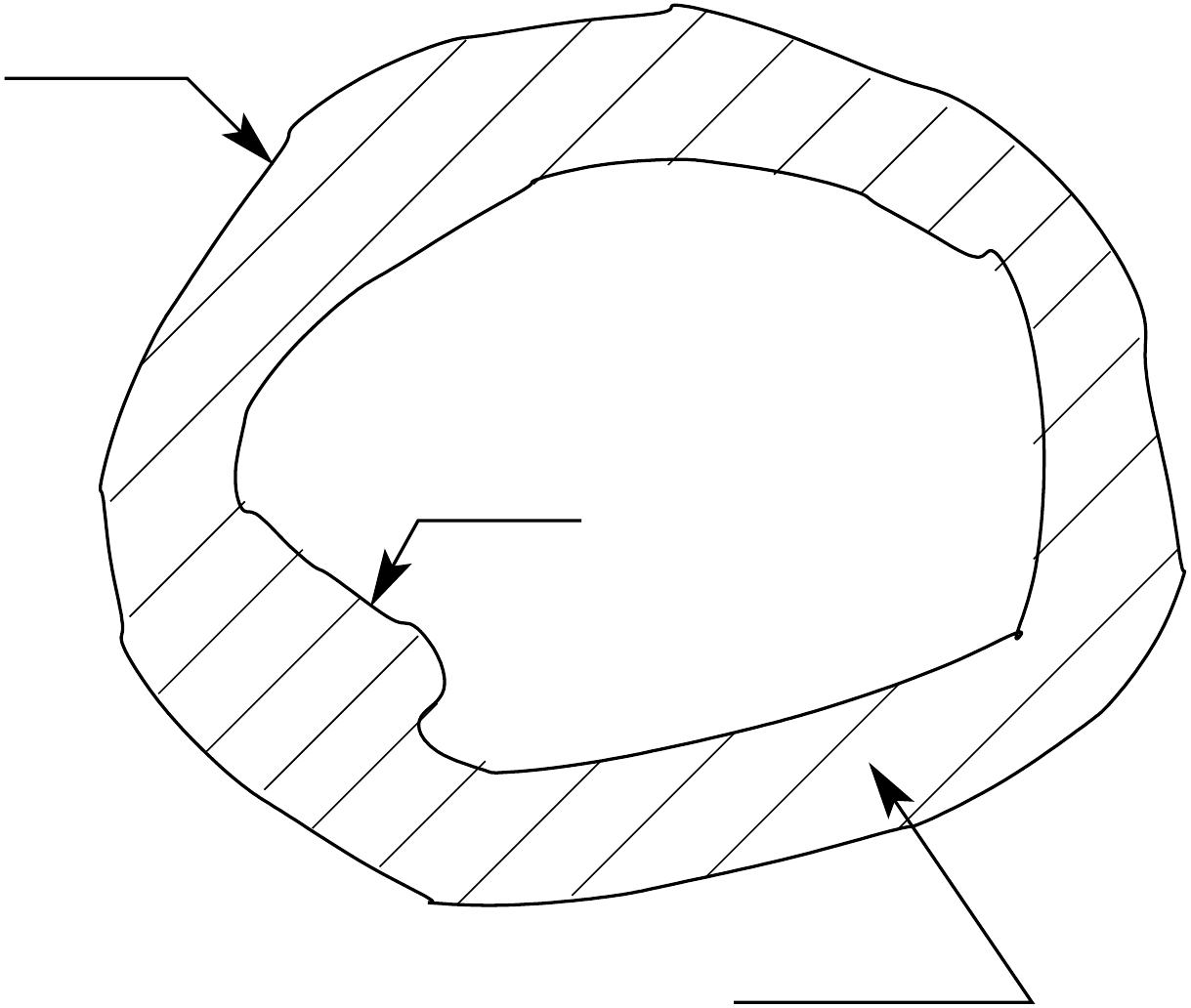}%
	\end{picture}%
	\setlength{\unitlength}{4144sp}%
	\begin{picture}(5566,4712)(2139,-6028)
	\put(2296,-1501){\makebox(0,0)[lb]{\smash{{\SetFigFont{22}{24.4}{\rmdefault}{\mddefault}{\updefault}{\color[rgb]{0,0,0}$\Gamma_D$}%
	}}}}
	\put(6706,-1591){\makebox(0,0)[lb]{\smash{{\SetFigFont{22}{24.4}{\rmdefault}{\mddefault}{\updefault}{\color[rgb]{0,0,0}$u=0$}%
	}}}}
	\put(4186,-3571){\makebox(0,0)[lb]{\smash{{\SetFigFont{22}{24.4}{\rmdefault}{\mddefault}{\updefault}{\color[rgb]{0,0,0}$\Gamma_N$}%
	}}}}
	\put(5671,-5821){\makebox(0,0)[lb]{\smash{{\SetFigFont{22}{24.4}{\rmdefault}{\mddefault}{\updefault}{\color[rgb]{0,0,0}$\Omega$}%
	}}}}
	\end{picture}%
}
\nopagebreak
\begin{figure}[H]
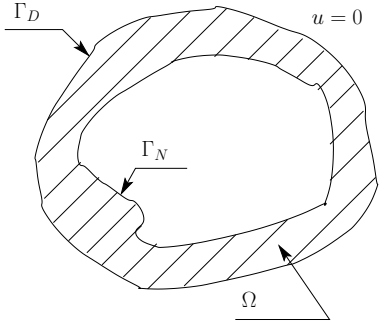

\caption{Slice of the elastic domain $\Omega$. The exterior boundary $\Gamma_D$ is immobile (or only subject to rigid displacements), while the boundary $\Gamma_N$ interacting with the blood is free, subject to the constraint of constant total volume.}
\label{fig0}
\end{figure}
\end{center}

\FloatBarrier

We are conscious that a complete modeling of such a problem would involve multiphysics considerations. More precisely, we would need, for example, to couple the bidomain model with hyperelasticity for describing the state equations~\cite{Dal2013, Mroue2019}. The analysis and computations related to such a coupled problem are both challenging topics. Due to the obvious complexity, we consider in this paper only hyperelasticity, and rather focus on the modeling and wellposedness questions that arise. Hyperelasticity is already a difficult topic from a mathematical point of view, therefore analysis is achieved with the help of an artificial damping in the hyperelastic model.

\paragraph{A damped hyperelastic model.} Given a smooth bounded domain $\Omega$ of $\R^d$ ($d=2$ or $3$), the state equation we consider in this paper is an elastodynamics system with damping and global volume preserving constraint. Its unknowns are a displacement field denoted by $u$, and a pressure $\mathfrak{p}$ depending only on the time variable, playing the role of a multiplier for taking into account this constraint. Data are represented by right-hand-sides $f$ and $g$, that represent volume forces in~$\Omega$ and boundary forces on~$\Gamma_N$, respectively. The initial state is given by the couple $(u_0,\udotini)$. The system that $(u,\mathfrak{p})$ is assumed to satisfy is given as follows: 
\begin{equation}
\begin{array}{rcl}
\displaystyle \ddot{u} -  \kappa\Delta \dot{u} -\divg \left((\I+\nabla u) \Sigma(u) \right) = f & & \text{in } \Omega \times (0,T),\\
\displaystyle\left(\kappa \frac{\p \dot{u}}{\p n} + (\I+\nabla u) \Sigma(u)\right)n + \mathfrak{p}\, \cof (\I +\nabla u)n = g & & 
\text{on } \Gamma_N \times (0,T),\\
\displaystyle\int_{\Omega} \det(\I+\nabla u)\, \d \Omega = 
\int_{\Omega} \det(\I+\nabla u_0)\, \d \Omega & & \text{in } (0,T), \\
\displaystyle  u = 0 & & \text{on } \Gamma_D \times (0,T), \\
\displaystyle u(\cdot,0) = u_0, \quad \dot{u}(\cdot,0) = \udotini & & \text{in } \Omega.
\end{array}\label{sysmain}
\end{equation}
The density of the material is assumed to be constant equal to~$1$ for simplicity. The strain energy model chosen for describing the elastic behavior determines the tensor field~$\Sigma$ (see section~\ref{sec-notation-ass} for details about the elasticity model). The damping term $\kappa\Delta \dot{u}$ (with $\kappa>0$ constant) is added for the sake of mathematical convenience. Indeed, it enables us to rely on a well-established framework for parabolic equations, while the original system is hyperbolic and nonlinear. The constraint in the third equation of~\eqref{sysmain}, namely the so-called Ciarlet-Ne\v{c}as global injectivity condition studied in~\cite{Necas1987}, reflects the fact that the global volume of the domain deformed by $\Id +u$ remains constant over time. Note that the homogeneous Dirichlet condition on~$\Gamma_D$ could be replaced by $u=0$ {\it up to rigid displacements}, leading us to consider solutions in quotient spaces. We explain how to derive system~\eqref{sysmain} in section~\ref{sec-derivPDE}. From section~\ref{sec-notation-ass} we will use the notation~$\sigma(\nabla u) = (\I+\nabla u) \Sigma(u)$ and~$\Phi(u) = \I+\nabla u$.

\paragraph{Wellposedness questions for nonlinear elasticity models.}
Without damping, and without volume constraints, global-in-time existence of elastic waves was obtained in~\cite{Sideris1996, Sideris1996bis, Agemi2000} assuming the so-called {\it null condition}, and in~\cite{Sideris2000} assuming that the stored energy satisfies a {\it nonresonance} condition. When considering the more classical incompressibility condition, namely the local constraint $\det(\I +\nabla u) \equiv 1$, the same question was investigated in~\cite{Ebin1993, Ebin1996, Becca-Thesis2003, Becca2005, Becca2007, Lei2016}, assuming that the data are small enough. While the system~\eqref{sysmain} satisfied by the displacement is written using the Lagrangian formalism, existence results were obtained in~\cite{Lei2015} (in the context of Hookean elasticity) and in~\cite{Yin2016} by adopting the Eulerian formalism. For convenience, we prefer to consider a parabolic regularization, represented by the dissipation term~$-\kappa \Delta \dot{u}$, orienting the study of system~\eqref{sysmain} towards methods related to parabolic equations. In this fashion, let us mention the contributions of~\cite{Zhang2009, Diagana2017, Vu2017, Rahmoune2017, Gou2020}. The parabolic framework is in particular suitable for realizing numerical simulations with finite element methods. As far as we know, the mathematical study of system~\eqref{sysmain} with the non-local constraint of constant global volume has never been addressed in the unsteady case.

\paragraph{Strategy.} 
It is reasonable to assume that the state variables are continuous in time, with values in smooth spaces. Therefore a strong functional framework is adopted, corresponding to the so-called $L^p$-maximal parabolic regularity~\cite{AMS2003}, leading us to assume that a solution $u$ of~\eqref{sysmain} satisfies
\begin{equation*}
\dot{u}\in \L^p(0,T; \WW^{2,p}(\Omega)) \cap \W^{1,p}(0,T;\LL^p(\Omega)),
\end{equation*}
with $p>3$. General assumptions are considered for the strain energy, and thus for the tensor operator~$\Sigma$. With the help of results obtained in~\cite{Pruess2002, Denk2006}, we first study system~\eqref{sysmain} linearized around $0$, leading to a wellposedness result for the nonhomogeneous linear system. Next, Lipschitz estimates are obtained for the nonlinearities, and we use a fixed-point method for proving existence and uniqueness of solutions for~\eqref{sysmain}, by assuming that $T>0$ is small enough. The difficulty lies in the careful and delicate derivation of these Lipschitz estimates in function of positive powers of~$T$. Finally, a continuation argument enables us to characterize the lifespan of the solutions. The main result is stated in Theorem~\ref{th-wellposed}.

\begin{remark} \label{remark-u0}
For the sake of convenience we will assume $u_0 = 0$, meaning that the initial deformation $\Id+u_0$ is the identity, and therefore that the reference configuration (described by $\Omega$, $\Gamma_D$ and $\Gamma_N$ as given in Figure~\ref{fig0}) corresponds to the initial configuration at $t=0$. Considering the more general case $u_0 \neq 0$ would enable us to choose a reference configuration as convenient as desired (in particular smooth), but would introduce other difficulties: Either system~\eqref{sysmain} would need to be linearized around~$u_0$, leading to a linear system with non-constant coefficients, or additional terms in the technical Lipschitz estimates of section~\ref{sec-Lip} would need to be taken into account. Therefore we choose $u_0=0$, preferring making assumptions on the initial geometric configuration (see section~\ref{sec-assgeo}). However, for numerical realizations, considering $u_0 \neq 0$ would be of interest, as it would enable us to choose a simple computational domain.
\end{remark}

\paragraph{Plan.} The paper is organized as follows: We introduce notation, assumptions and functional setting in section~\ref{sec-notation}. In particular, in section~\ref{sec-strain} we give examples of class of strain energies that fulfill the assumptions that we introduce in section~\ref{sec-notation-ass}. In section~\ref{sec-well} we study the linearized system in the context of the $L^p$-maximal parabolic regularity, and deduce in Corollary~\ref{coro-linNH} a compact estimate for the nonhomogenenous system. Section~\ref{sec-well-main} is devoted to the local-in-time existence and uniqueness of solutions for system~\eqref{sysmain}. We prove Lipschitz estimates in section~\ref{sec-Lip}, and derive the main result in section~\ref{sec-mainTh}, namely Theorem~\ref{th-locexist}. In Appendix~\ref{sec-modeling}, we present modeling aspects of the problem, in particular how to derive system~\eqref{sysmain} from the least-action principle. Appendix~\ref{sec-app-B} contains the proof of a technical lemma.

\subsection*{Acknowledgments}
The author thanks Prof. Karl Kunisch (University of Graz \& RICAM, Linz), Prof. Gernot Plank (Biotechmed Graz) and his research group. The discussions and ideas he had during his stay in Graz have lead to the present article. 

\section{Notation and preliminaries} \label{sec-notation}

The notation presented here will be used throughout the rest of the paper. The reader is therefore invited to refer to the present section while browsing the other sections. 

\subsection{Linear Algebra notation}

The inner product between two vectors $u$, $v \in \R^d$ is denoted as $u\cdot v$, the corresponding Euclidean norm as $|u|_{\R^d}$, and the tensor product is denoted by $u\otimes v \in \R^{d\times d}$, such that $(u\otimes v)_{ij} := u_i v_j$. The inner product between two matrices $A$, $B\in \R^{d\times d}$ is denoted as $A:B = \mathrm{trace}(A^TB)$, and the associated Euclidean norm satisfies $| AB |_{\R^{d\times d}} \leq | A|_{\R^{d\times d}} |B|_{\R^{d\times d}}$. The tensor product between matrices is denoted as $A \otimes B \in \R^{d\times d \times d \times d}$, such that for all matrix $C \in \R^{d\times d}$, $(A\otimes B)C := (B:C)A \in \R^{d\times d}$. We denote by $\cof (A)$ the cofactor matrix of any matrix field $A$. Recall that this is a polynomial function of the coefficients of $A$. When $A$ is invertible, the following formula holds
\begin{equation*}
\cof (A) = (\det (A))A^{-T} .
\end{equation*}
Recall that $H \mapsto (\cof A):H$ is the differential of $A \mapsto \det(A)$ at point $A$. In order to distinguish scalar fields, vector fields and matrix fields, we will use the following type of notation
\begin{equation*}
\L^{p}(\Omega) = \left\{\varphi :\Omega \rightarrow \R \mid \int_{\Omega} |\varphi|_{\R}^p\d \Omega <\infty \right\}, \quad
\LL^{p}(\Omega) = [\L^p(\Omega)]^d, \quad
\LLL^{p}(\Omega) = [\L^p(\Omega)]^{d\times d},
\end{equation*}
that we transpose by analogy to other kinds of Lebesgue, Sobolev and Besov spaces.

\subsection{Geometric assumptions and functional spaces} \label{sec-assgeo}
The domain $\Omega \subset \R^2$ or $\R^3$  will be assumed smooth and bounded. We consider two open disjoint subsets $\Gamma_D, \Gamma_N \subset \p \Omega$ such that $\overline{\Gamma_D} \cup \overline{\Gamma_N} = \p \Omega$, assumed to be nontrivial in the sense that their surface Lebesgue measures satisfy $|\Gamma_D| > 0$ and $|\Gamma_N|>0$, and smooth in the sense that the surfaces $\Gamma_D$ and $\Gamma_N$ are {\it regular}, namely at any point of $\Gamma_D$ and $\Gamma_N$ we can define a normal vector. More specifically we assume that the outward unit normal denoted by $n$ lies in $\WW^{2-1/p,p}(\Gamma_N)$ on $\Gamma_N$, which implies in particular $n\in\HH^{1/2}(\Gamma_N)$. Furthermore we assume that $\Gamma_N$ is closed, so that~$\HH^{1/2}(\Gamma_N)' = \HH^{-1/2}(\Gamma_N)$.

We have introduced fractional Sobolev spaces. Recall the definition of~$\W^{\alpha,p}(D)$ for any smooth bounded domain $D\subset \R^n$ when $\alpha \in (0,1)$ (see~\cite{DiNezza}):
\begin{equation}
\varphi \in \W^{\alpha,p}(D) \Leftrightarrow
\|\varphi\|_{\W^{\alpha,p}(D)} := \left(
\| \varphi\|_{\L^p(D)}^p + \int_D\int_D 
\frac{|\varphi(x)-\varphi(y)|^p}{|x-y|^{n+\alpha p}}\d x\d y
\right)^{1/p} <\infty.
\label{def-fractional}
\end{equation}
It is well-known that $\W^{\gamma,p}(0,T;\R)$ is continuously embedded in $\mathcal{C}([0,T];\R)$ when $\gamma p>1$. A proof is provided in~\cite[Theorem~8.2]{DiNezza}, relying on~\cite[Lemma~2.2]{Giusti} which introduces an Hardy-Littlewood maximal function. As stated in these references, the constant of this embedding can {\it a priori} depend on the size of the domain~$(0,T)$, and thus on~$T$, and may possibly increase when $T$ decreases. Let us derive carefully the dependence with respect to~$T$ of the constant of this embedding, by providing a self-consistent proof in section~\ref{sec-app-B}.

\begin{lemma} \label{lemma-DiNezza}
Let be $B$ a Banach space, and $\varphi \in \W^{\gamma,p}(0,T;\B)$ for some $1/p<\gamma <1$. Then for all $t\in (0,T]$ we have
\begin{equation*}
| \varphi(t) -\varphi(0) | \leq CT^{\gamma-1/p} \|\varphi\|_{\W^{\gamma,p}(0,T;B)},
\end{equation*}
where the constant $C>0$ depends only on $p$ and $\gamma$. Further, we have
\begin{equation*}
\|\varphi\|_{\L^{\infty}(0,T;B)} \leq |\varphi(0)| + 
CT^{\gamma-1/p} \|\varphi\|_{\W^{\gamma,p}(0,T;B)}.
\end{equation*}
\end{lemma}

Throughout the paper, we denote by $C$ any generic positive constant depending only on $\Omega$, the exponent $p$ and $\kappa$, in particular it is independent of~$T$.

We will assume that $p>3$. Given $T>0$, the displacement variable $u$ and its time-derivative $\dot{u}$ will be considered in the spaces defined below:
\begin{equation*}
\begin{array} {l}
 u \in \mathcal{U}_{p,T}(\Omega)  :=  \W^{1,p}(0,T; \WW^{2,p}(\Omega)\cap \WW^{1,p}_{0,D}(\Omega)) 
\cap \W^{2,p}(0,T; \LL^p(\Omega )), \\
 \dot{u} \in \dot{\mathcal{U}}_{p,T}(\Omega)  :=  
\L^p(0,T; \WW^{2,p}(\Omega)\cap \WW^{1,p}_{0,D}(\Omega)) 
\cap \W^{1,p}(0,T;\LL^p(\Omega)),
\end{array}
\end{equation*}
where~$\WW^{1,p}_{0,D}(\Omega)$ is the space of functions~$v \in \WW^{1,p}(\Omega)$ satisfying~$v_{|\Gamma_D}=0$. We denote by $p'$ the dual exponent of $p$ satisfying $1/p+1/{p'} = 1$. Using the notation of~\cite{Triebel}, the trace space for $\dot{u} \in \dot{\mathcal{U}}_{p,T}(\Omega)$ involves the Besov spaces obtained by real interpolation as $\displaystyle \left(\LL^p(\Omega);\WW^{2,p}(\Omega)\right)_{1/{p'},p} =: \BB^{2/{p'}}_{pp}(\Omega)$ and $\big(\LL^p(\Omega) ;\WW^{1,p}_{0,D}(\Omega)\big)_{1/{p'},p} =: \mathring{\BB}^{1/{p'}}_{pp}(\Omega)$, which coincide with $\WW^{2/{p'},p}(\Omega)$ and $\WW^{1/{p'},p}_{0,D}(\Omega)$, respectively. Therefore the initial conditions will be assumed to lie in the trace space of $\mathcal{U}_{p,T} \times \dot{\mathcal{U}}_{p,T}(\Omega)$, namely:
\begin{equation*}
(u_0,\udotini) \in \mathcal{U}_p^{(0,1)}(\Omega)  := 
\left(\WW^{2,p}(\Omega)\cap \WW^{1,p}_{0,D}(\Omega)  \right)
\times
\left(\WW^{2/{p'},p}(\Omega) \cap \WW^{1/{p'},p}_{0,D}(\Omega)\right).
\end{equation*}
Note that the space $\mathcal{U}_p^{(0,1)}(\Omega)$ coincides with the set $
\left\{		(u(0),\dot{u}(0)) \mid u \in \mathcal{U}_{p,T}(\Omega)\times \dot{\mathcal{U}}_{p,T}(\Omega)	\right\}$. We refer to~\cite{Chill2005} and~\cite[section~6]{Arendt2007} for more details. Actually, as explained in Remark~\ref{remark-u0}, we will assume $u_0 \equiv 0$ for the sake of convenience. Note that since for $p>d$ we have $2/{p'} \in (1,2)$, the embedding $\WW^{2/{p'},p}(\Omega) \hookrightarrow \WW^{1,p}(\Omega)$ hold, and therefore~$\dot{u}_0$ is continuous on~$\Omega$. 
Recall that for $p>d$, the space $\WWW^{1,p}(\Omega)$ is an algebra. In particular, there exists a positive constant $C$ such that for all $A, \, B \in \WWW^{1,p}(\Omega)$, we have 
\begin{equation}
\|AB \|_{\WWW^{1,p}(\Omega)}  \leq 
C \| A \|_{\WWW^{1,p}(\Omega)} \|B  \|_{\WWW^{1,p}(\Omega)}.
\label{W-algebra}
\end{equation}
See for example~\cite[Lemma~A.1]{BB1974}. Therefore the operator $\WW^{2,p}(\Omega)\ni u \mapsto \cof(\I+\nabla u)$ has values in $\WW^{1,p}(\Omega)$. We refer to Lemma~\ref{lemma-cof} for Lipschitz properties of this operator. The volume right-hand-side will be considered as follows:
\begin{equation*}
f \in \mathcal{F}_{p,T}(\Omega) := \L^p(0,T;\LL^p(\Omega)).
\end{equation*}
Let us specify the space of the Neumann boundary condition. Following~\cite{Pruess2002}, the second equation of~\eqref{sysmain} is considered for
\begin{equation*}
g \in \mathcal{G}_{p,T}(\Gamma_N) := 
\L^p(0,T;\WW^{1-1/p,p}(\Gamma_N)) \cap
\W^{1/2-1/{(2p)},p}(0,T; \LL^p(\Gamma_N))
\end{equation*}
satisfying the compatibility condition $\displaystyle \kappa \frac{\p \dot{u}_0}{\p n} + \Sigma(0)n = g(\cdot,0)$. Indeed, when $p>3$ this space is embedded in~$\mathcal{C}([0,T]; \LL^p(\Gamma_N))$. We equip~$\mathcal{G}_{p,T}(\Gamma_N)$ with the following norm
\begin{equation*}
\|g\|_{\mathcal{G}_{p,T}(\Gamma_N)} :=
\|g \|_{\L^p(0,T;\WW^{1-1/p,p}(\Gamma_N))} 
+ \|g\|_{\W^{1/2-1/{(2p)},p}(0,T; \LL^p(\Gamma_N))}
+ \|g\|_{\L^{\infty}(0,T; \LL^p(\Gamma_N))}.
\end{equation*}
Besides, since we have $\cof(\I+\nabla u) n \in \W^{1,p}(0,T; \WW^{1-1/p,p}(\Gamma_N))$, we can deduce that for $p>3$ the pressure is continuous in time. We refer to section~\ref{sec-pressure} for more details on the regularity of~$\mathfrak{p}$. 
More specifically, the pressure~$\mathfrak{p}$ is considered in the following space
\begin{equation*}
\mathcal{P}_{p,T} : = \W^{1/2-1/(2p),p}(0,T;\R),
\end{equation*}
that we equip with the norm
\begin{equation*}
\|\mathfrak{p}\|_{\mathcal{P}_{p,T}} := 
\|\mathfrak{p}\|_{\W^{1/2-1/(2p),p}(0,T;\R)}
+\|\mathfrak{p}\|_{\L^{\infty}(0,T;\R)}.
\end{equation*}
In the same fashion, the space in which we will consider the global injectivity constraint (third equation of~\eqref{sysmain}) derived in time (see system~\eqref{sysmain2}) is denoted by
\begin{equation*}
\mathcal{H}_{p,T} := \W^{1-1/(2p),p}(0,T;\R),
\end{equation*}
and we equip it with the norm
\begin{equation*}
\| h \|_{\mathcal{H}_{p,T}} := \|h\|_{\W^{1-1/(2p),p}(0,T;\R)}
+\|h\|_{\L^{\infty}(0,T;\R)}.
\end{equation*}
Finally, recall the boundary trace embedding~\cite[Lemma~3.5]{Denk2006}, namely
\begin{equation}
\begin{array} {l}
\|v_{|\Gamma_N}\|_{\W^{1-1/(2p),p}(0,T;\LL^p(\Gamma_N))\cap\L^p(0,T;\WW^{2-1/p,p}(\Gamma_N))}
\leq C\|v \|_{\dot{\mathcal{U}}_{p,T}(\Omega)},
\end{array}
\label{est-trace-emb}
\end{equation}
where the constant $C>0$ is independent of~$T$. The same result also provides the following estimates on the normal derivative
\begin{equation}
\begin{array} {rcl}

\displaystyle
\left\|\frac{\p v}{\p n}\right\|_{\W^{1/(2p'),p}(0,T;\LL^p(\Gamma_N))}
& \leq & C\|v \|_{\dot{\mathcal{U}}_{p,T}(\Omega)},\\
\displaystyle
\left\|\frac{\p v}{\p n}\right\|_{\W^{1/(2p'),p}(0,T;\HH^{1/2-1/p}(\Gamma_N)')}
& \leq & C\|v \|_{\L^p(0,T;\HH^2(\Omega)) \cap \W^{1,p}(0,T;\LL^2(\Omega))}
\end{array}
\label{est-trace-emb2}
\end{equation}
where $C>0$ is again independent of~$T$. The second estimate is actually deduced directly by interpolation.

\subsection{Assumptions on the strain energy} \label{sec-notation-ass}
We introduce the deformation gradient tensor associated with a displacement field $u$ as follows:
\begin{equation*}
\Phi(u) = \I + \nabla u.
\end{equation*}
The strain energy of the elastic material is denoted by $\mathcal{W}$, and is assumed to be a function of the Green--Saint-venant strain tensor
\begin{equation*}
E(u) := \frac{1}{2}\left(\Phi(u)^T \Phi(u) - \I\right) = 
\frac{1}{2}\left((\I+\nabla u)^T(\I+\nabla u) - \I \right).
\end{equation*}
We denote classically~\cite{Ciarlet} by $\check{\Sigma}$ the differential of~$\mathcal{W}$:
\begin{equation*}
\check{\Sigma}(E) = \frac{\p \mathcal{W}}{\p E}(E).
\end{equation*}
In system~\eqref{sysmain}, the tensor $\Sigma$ is related to the strain energy $\mathcal{W}$ as
\begin{equation*}
\Sigma(u) := \frac{\p \mathcal{W}}{\p E}(E(u)).
\end{equation*}
The derivation of system~\eqref{sysmain} from $\mathcal{W}$ is given in section~\ref{sec-derivPDE}. We make the following assumptions on~$\mathcal{W}$:
\begin{description}
\item[$\mathbf{A1}$] 
The Nemytskii operator $\mathcal{W}: \WWW^{1,p}(\Omega) \ni E \mapsto \mathcal{W} (E) \in \R$ is twice continuously Fr\'echet-differentiable.
\item[$\mathbf{A2}$] For all symmetric matrix $E \in \R^{d\times d}$, the matrix~$\check{\Sigma}(E)$ is symmetric too.
\end{description}
In section~\ref{sec-strain} we show that well-known examples of strain energies from the literature fulfill this assumption. From section~\ref{sec-well} we will use the notation
\begin{equation*}
\sigma(\nabla u) := (\I + \nabla u)\Sigma(u) = 
 (\I + \nabla u) \frac{\p \mathcal{W}}{\p E}(E(u)).
\end{equation*}
Note that the tensor $\sigma$ is a function of $\nabla u$ only, since the strain energy is chosen to be a function of he Green -- St-Venant strain tensor $E(u) = \displaystyle \frac{1}{2}(\nabla u + \nabla u^T + \nabla u^T \nabla u )$. Assumption~$\mathbf{A1}$ implies that~$\sigma$ is locally Lipschitz from~$\WWW^{1,p}(\Omega)$ to~$\WWW^{1,p}(\Omega)$. This is useful for deriving in Lemma~\ref{lemma-sigmax} Lipschitz estimates on the term $\sigma(\nabla u)$. For our purpose, namely proving Theorem~\ref{th-wellposed}, it would be sufficient to assume that~$\mathcal{W}$ is only of class~$\mathcal{C}^{1,\delta}$ for some~$\delta\in(0,1)$, implying that~$\sigma$ is~$\delta$-H\"{o}lder, but in view of the examples given in section~\ref{sec-strain}, it is reasonable to simply consider Assumption~$\mathbf{A1}$, avoiding to deal with pointless technicalities.

Therefore Assumption~$\mathbf{A1}$ allows us to calculate the derivative of~$\sigma$, denoted by~$\sigma'(\nabla u) \in \mathscr{L}(\WWW^{1,p}(\Omega);\WWW^{1,p}(\Omega))$ at point~$u$, as follows
\begin{equation}
\sigma'(\nabla u).(\nabla v)  = 
\displaystyle(\nabla v) \Sigma(u) + 
\Phi(u)\left(\frac{\p^2 \mathcal{W}}{\p E^2}(E(u)).(E'(u).v) \right),
\label{sigmaprime}
\end{equation}
where $E'(u).v = \frac{1}{2} \left(\Phi(u)^T\nabla v + (\nabla v)^T \Phi(u) \right)$. Assumption~$\mathbf{A1}$ implies that the mapping~$ \WW^{2,p}(\Omega) \ni u \mapsto  \sigma'(\nabla u) \in \mathscr{L}(\WWW^{1,p}(\Omega);\WWW^{1,p}(\Omega))$ is continuous. Assumption~$\mathbf{A2}$ is used in section~\ref{sec-derivPDE} only.

\subsection{Examples of strain energies} \label{sec-strain}
Let us show that the assumptions~$\mathbf{A1}$-$\mathbf{A2}$ are satisfied by classical elasticity models.

\paragraph{The Saint Venant-Kirchhoff's model.} It corresponds to the following strain energy
\begin{equation*}
\mathcal{W}_1(E) = \mu_L \trace\left( E^2 \right) + \frac{\lambda_L}{2} \trace(E)^2,
\end{equation*}
where $\mu_L >0$ and $\lambda_L \geq 0$ are the so-called Lam\'e coefficients. The energy is clearly twice differentiable, its first- and second-derivatives of $\mathcal{W}_1$ are given respectively by
\begin{equation*}
	\check{\Sigma}_1(E) = 2\mu_L E + \lambda_L \trace(E)\I, \qquad
	\frac{\p \check{\Sigma}_1}{\p E}(E)  =  2\mu_L \III + \lambda_L \I \otimes \I,
\end{equation*}
where $\III \in \R^{d\times d \times d \times d}$ denotes the identity tensor of order~$4$. In particular, we see that if the matrix~$E$ is symmetric, then~$\check{\Sigma}_1$ defines a symmetric matrix.

\paragraph{The Fung's model.} It corresponds to the following strain energy
\begin{equation*}
	\mathcal{W}_2(E) = \mathcal{W}_2(0) + \beta \left(\exp\left(\gamma \ \trace(E^2)\right) - 1\right),
\end{equation*}
where $\mathcal{W}_2(0) \geq 0$, $\beta >0$ and $\gamma > 0$ are given coefficients. The space $\WWW^{1,p}(\Omega)$ is invariant under composition of the exponential function when $p>d$ (see~\cite{BB1974}, Lemma~A.2. page 359). The first- and second-derivatives of $\mathcal{W}_2$ are given respectively by
\begin{equation*}
	\check{\Sigma}_2(E)  =  2\gamma\beta \exp\left(\gamma \ \trace(E^2) \right) E, \qquad
	\frac{\p \check{\Sigma}_2}{\p E}(E)  =  \beta \exp\left(\gamma \ \trace(E^2) \right)
	\left( 2\gamma \III + (2\gamma)^2 E \otimes E\right).
\end{equation*}
Again, if~$E$ is symmetric, then~$\check{\Sigma}_2$ is symmetric.

\paragraph{The Ogden's model.} The family of strain energies corresponding to this model are linear combinations of energies of the following form
\begin{equation*}
	\mathcal{W}_3(E)  =   \trace\left((2E+\I)^{\gamma} - \I \right),
\end{equation*}
where $\gamma \in \R$. Since the tensor $2E+\I$ is real and symmetric, the expression $(2E+\I)^{\beta} $ makes sense for any $\beta \in \R$ by diagonalizing~$2E+\I$, and the energy $\mathcal{W}_3(E)$ can be expressed in terms of the eigenvalues of $2E+\I$. Since $2E(u)+\I = (\I+\nabla u)^T(\I+\nabla u)$, if $(\lambda_i)_{1\leq i \leq d}$ denote the singular values of $\I+\nabla u$, and $(\mu_i)_{1\leq i \leq d}$ denote those of $E(u)$, we have
\begin{equation*}
	\mathcal{W}_3(E)  =  \sum_{i=1}^d \left(\lambda_i^{2\gamma}  -1\right)
	= \sum_{i=1}^d \left((1+2\mu_i)^\gamma -1\right), \qquad
	\check{\Sigma}_3(E) =  2\gamma (2E+\I)^{\gamma-1}.
\end{equation*}
Denoting by $(v_l)_{1\leq l\leq d}$ the normalized orthogonal eigenvectors of $E$, we further write
\begin{equation*}
\check{\Sigma}_3(E)  = \sum_{i=1}^d 2\gamma(2\mu_i+1)^{\gamma-1}v_i \otimes v_i.
\end{equation*}
Note that the operator $v_i \otimes v_i$ is the projection on $\mathrm{Span}(v_i)$, and Assumption~$\mathbf{A2}$ is satisfied by~$\check{\Sigma}_3$. The sensitivity of the eigenvalues and eigenvectors with respect to the matrix can be derived for example from~\cite{Stewart} (see Theorem~IV.2.3 page~183, and Remark~2.9 page 239). Thus after calculations we get
\begin{equation*}
\begin{array} {rcl}
\displaystyle \frac{\p \check{\Sigma}_3}{\p E}(E) & = & \displaystyle 2\gamma
\sum_{i=1}^d 2(\gamma-1)(2\mu_i+1)^{\gamma-2} (v_i\otimes v_i) \otimes (v_i\otimes v_i) \displaystyle \\
& & + \displaystyle  2\gamma\sum_{i=1}^d (2\mu_i+1)^{\gamma-1} 
\sum_{j\neq i} \frac{1}{\mu_i-\mu_j}(v_j\otimes v_i) \otimes (v_j\otimes v_i) .
\end{array}
\end{equation*}
This expression shows that the strain energy $\mathcal{W}_3$ fulfills also Assumption~$\mathbf{A1}$.

\section{On the linearized system} \label{sec-well}

Let be $T>0$. The goal of this section is to study system~\eqref{sysmain} linearized around $(\dot{u},\mathfrak{p}) = (0,0)$. Note that the global injectivity constraint and the homogeneous Dirichlet condition on~$\Gamma_D$, namely the third and fourth equations of system~\eqref{sysmain} respectively, can be derived in time, leading to the following system (see section~\ref{sec-modeling1} for more details):
\begin{equation}\label{sysmain2}
\begin{array} {rcl} 
\ddot{u} - \kappa \Delta \dot{u}
- \divg \sigma(\nabla u) = f & & \text{in $\Omega \times(0,T)$},  \\
\displaystyle \kappa \frac{\p \dot{u}}{\p n} + \sigma(\nabla u)n 
 + \mathfrak{p}\, \cof\left(\Phi(u) \right) n = g & & \text{on $\Gamma_N\times(0,T)$}, 
\\
\displaystyle \int_{\Gamma_N} \dot{u} \cdot \cof\left(\Phi(u) \right)n\, \d\Gamma_N = 0 & & \text{in } (0,T) \\
\dot{u} = 0 & & \text{on $\Gamma_0\times(0,T)$},  \\
u(\cdot, 0) = 0, \displaystyle \quad \dot{u}(\cdot, 0) = \udotini & & \text{in $\Omega$}.
\end{array}
\end{equation}
In what follows, the variable $v$ will play the role of $\dot{u}$, and therefore we will consider
\begin{equation*}
v \in \dot{\mathcal{U}}_{p,T}(\Omega) = 
\L^p(0,T;\WW^{2,p}(\Omega)\cap \WW^{1,p}_{0,D}(\Omega)) \cap \W^{1,p}(0,T;\LL^p(\Omega)).
\end{equation*}
We rewrite~\eqref{sysmain2} in terms of $v$, by splitting the linear part and the remainder as follows
\begin{equation} \label{sysmain3}
\begin{array}{rcl}
\dot{v} - \kappa \Delta v
= f+\divg \sigma(\nabla u) & & \text{in $\Omega \times(0,T)$},  \\
\displaystyle \kappa \frac{\p v}{\p n} + \mathfrak{p}\, n = 
g - \sigma(\nabla u)n  + \mathfrak{p}\, (\I - \cof\left(\Phi(u) \right))n
 & & \text{on $\Gamma_N\times(0,T)$}, \\[5pt]
\displaystyle \int_{\Gamma_N} v \cdot n\, \d\Gamma_N = 
\int_{\Gamma_N} v \cdot \left(\I-\cof\left(\Phi(u) \right)\right)n\, \d\Gamma_N
 & & \text{in } (0,T) \\
v = 0 & & \text{on $\Gamma_0\times(0,T)$},  \\
\displaystyle u(\cdot, t) = \int_0^t v(\cdot,s)\d s, 
\displaystyle \quad v(\cdot, 0) = \udotini & & \text{in $\Omega$}.
\end{array}
\end{equation}
We will assume the compatibility conditions~$\displaystyle \kappa\frac{\p \dot{u}_0}{\p n} + \sigma(0)n = g(\cdot,0)$ on~$\Gamma_N$ and $\displaystyle \int_{\Gamma_N} \dot{u}_0 \cdot n\, \d \Gamma_N = 0$. We want to prove existence of solutions for system~\eqref{sysmain3} in the context of the $L^p$-maximal regularity~\cite{Denk2006}, namely solutions $v$ that lie in the space~$\dot{\mathcal{U}}_{p,T}(\Omega)$. Recall roughly the rigidity of this functional framework: The fact that if an operator owns the $L^p$-maximal regularity property for a given time exponent~$p\in (1,\infty)$ (namely the exponent~$p$ that appears in $\L^p(0,T)$ and $\W^{1,p}(0,T)$ spaces), then it also owns this property for any~$p\in(1,\infty)$. Further, this property is also independent of $T \in (0,\infty)$.\\ 
We rely on the main result of~\cite{Pruess2002}, stating maximal parabolic regularity in $\dot{\mathcal{U}}_{p,T}(\Omega)$ for a linear parabolic problem with mixed boundary conditions. Keep in mind that $p>3$. We recall and adapt this result to our context in Proposition~\ref{prop-Pruess}.
 
\begin{proposition} \label{prop-Pruess}
Assume that 
\begin{equation*}
F \in \mathcal{F}_{p,T}(\Omega), \quad 
G \in \mathcal{G}_{p,T}(\Gamma_N), \quad
v_0 \in \mathbf{W}^{2/{p'},p}(\Omega) \cap \mathbf{W}^{1/{p'},p}_{0,D}(\Omega),
\end{equation*}
with the compatibility condition $\displaystyle \kappa\frac{\p v_0}{\p n} = G(\cdot,0)$ on $\Gamma_N$. Then there exists a unique solution $v \in \dot{\mathcal{U}}_{p,T}(\Omega)$ to the following system
\begin{equation}
\left\{ \begin{array} {rcl}
\displaystyle \dot{v} - \kappa \Delta v = F & & \text{in } \Omega \times (0,T), \\
\displaystyle \kappa \frac{\p v}{\p n} = G & & \text{on } \Gamma_N \times (0,T), \\
v = 0 & & \text{on } \Gamma_D \times (0,T), \\
v(\cdot,0) = v_0 & & \text{in } \Omega.
\end{array} \right. \label{sys-Pruess}
\end{equation}
Moreover, there exists a constant $C(T)>0$ such that
\begin{equation}
\|v\|_{\dot{\mathcal{U}}_{p,T}(\Omega)}  \leq  C(T) \left(
\|v_0\|_{\mathbf{W}^{2/{p'},p}(\Omega)} + 
\|F\|_{\mathcal{F}_{p,T}(\Omega)} + 
\|G\|_{\mathcal{G}_{p,T}(\Gamma_N)}
\right).
\label{est-Pruess}
\end{equation}
In particular, the constant $C(T)$ is non-decreasing with respect to~$T$.
\end{proposition}

In the rest of the paper the notation~$C(T)$ will refer generically to a positive constant depending only on~$\Omega$, the exponent~$p$, the regularizing coefficient $\kappa$ and non-decreasingly on~$T$.

\begin{proof}
This is a consequence of~\cite[Theorem~4.3]{Pruess2002}.
\end{proof}

In order to introduce the pressure variable, we need the following result:

\begin{lemma} \label{lemma-const}
Assume $n\in \HH^{1/2}(\Gamma_N)$, and that $g \in \HH^{-1/2}(\Gamma_N)$ satisfies $\langle g ; \varphi \rangle_{\HH^{-1/2}(\Gamma_N);\HH^{1/2}(\Gamma_N)} = 0$ for all $\varphi \in \HH^{1/2}(\Gamma_N)$ such that $\displaystyle \langle \varphi ; n \rangle_{\LL^2(\Gamma_N)} = 0$. 
Then there exists $\mathfrak{p}\in \R$ such that $g= \mathfrak{p}\, n$ in $\HH^{1/2}(\Gamma_N)$.
\end{lemma}

\begin{proof}
We set $\mathfrak{p} = \displaystyle \frac{1}{|\Gamma_N|}\langle g ; n\rangle_{\HH^{-1/2}(\Gamma_N),\HH^{1/2}(\Gamma_N)}
$, and for all $\varphi \in \HH^{1/2}(\Gamma_N)$ we calculate
\begin{eqnarray*}
\langle g -\mathfrak{p}\, n ; \varphi \rangle_{\HH^{-1/2}(\Gamma_N),\HH^{1/2}(\Gamma_N)} & = & \langle g ; \varphi \rangle_{\HH^{-1/2}(\Gamma_N),\HH^{1/2}(\Gamma_N)} - \mathfrak{p}\langle \varphi; n \rangle_{\LL^2(\Gamma_N)} \\
& = & \langle g ; \varphi \rangle_{\HH^{-1/2}(\Gamma_N),\HH^{1/2}(\Gamma_N)} - \frac{1}{|\Gamma_N|} 
\langle g ; n \rangle_{\HH^{-1/2}(\Gamma_N),\HH^{1/2}(\Gamma_N)}
\langle \varphi ; n \rangle_{\LL^2(\Gamma_N)} \\
& = & \left\langle g ; \varphi - \frac{1}{|\Gamma_N|} \langle \varphi ; n \rangle_{\LL^2(\Gamma_N)} n \right\rangle_{\HH^{-1/2}(\Gamma_N),\HH^{1/2}(\Gamma_N)}.
\end{eqnarray*}
Since $\tilde{\varphi} := \displaystyle  \varphi - \frac{1}{|\Gamma_N|} \langle \varphi ; n \rangle_{\LL^2(\Gamma_N)} n$ satisfies $\langle \tilde{\varphi} ; n \rangle_{\LL^2(\Gamma_N)} = 0$, by assumption we deduce $\langle g -\mathfrak{p}\, n ; \varphi \rangle_{\HH^{-1/2}(\Gamma_N),\HH^{1/2}(\Gamma_N)} = 0$ for all $\varphi\in \HH^{1/2}(\Gamma_N)$, namely $g-\mathfrak{p}\, n = 0$ in $\HH^{-1/2}(\Gamma_N)$, and thus $g=\mathfrak{p}\, n \in \HH^{1/2}(\Gamma_N)$, which completes the proof.
\end{proof}

We deduce the same result as Proposition~\ref{prop-Pruess} when the solution needs to satisfy an additional constraint.

\begin{proposition} \label{prop-pressure}
Given the assumptions of Proposition~\ref{prop-Pruess}, there exists a unique couple $(v,\mathfrak{p}) \in \dot{\mathcal{U}}_{p,T}(\Omega) \times \mathcal{P}_{p,T}$ to the following system:
\begin{subequations}\label{sys-Pruess-constraint}
\begin{eqnarray}
\displaystyle \dot{v} - \kappa \Delta v = F & & \text{in } \Omega \times (0,T), 
\label{sys-P-1}\\
\displaystyle \kappa \frac{\p v}{\p n} +\mathfrak{p}\, n = G & & \text{on } \Gamma_N \times (0,T), 
\label{sys-P-2}\\
\displaystyle \int_{\Gamma_N} v\cdot n \, \d \Gamma_N = 0 & & \text{in } (0,T), 
\label{sys-P-const}\\
v = 0 & & \text{on } \Gamma_D \times (0,T), \\
v(\cdot,0) = v_0 & & \text{in } \Omega.
\end{eqnarray}
\end{subequations}
It satisfies:
\begin{equation}
\|v\|_{\dot{\mathcal{U}}_{p,T}(\Omega)} + 
\|v\|_{\L^{\infty}(0,T;\LL^p(\Gamma_N))} +
\| \mathfrak{p}\|_{\mathcal{P}_{p,T}} 
 \leq  C(T) \left(
\|v_0\|_{\mathbf{W}^{2/{p'},p}(\Omega)} + 
\|F\|_{\mathcal{F}_{p,T}(\Omega)} + 
\|G\|_{\mathcal{G}_{p,T}(\Gamma_N)}
\right),
\label{est-Pruess2}
\end{equation}
where $C(T)>0$ denotes a generic constant which is non-decreasing with respect to~$T$.
\end{proposition}

\begin{proof}
Define the following functional spaces and the operator $A$:
\begin{equation*}
\begin{array} {l}
 \mathcal{V}(\Omega) =  \left\{ v \in \HH^1(\Omega) \mid v_{|\Gamma_D} = 0, \ \displaystyle \int_{\Gamma_N} v\cdot n \, \d \Gamma_N = 0 \right\}, 
\quad \mathcal{V}(\Gamma_N) =  \left\{ v \in \HH^{1/2}(\Gamma_N) \mid \displaystyle \int_{\Gamma_N} v\cdot n \, \d \Gamma_N = 0 \right\}, \\[5pt]
\langle Av ; \varphi \rangle_{\mathcal{V}(\Omega)', \mathcal{V}(\Omega)} = \displaystyle
\kappa\int_{\Omega} \nabla v : \nabla \varphi \, \d \Omega.
\end{array}
\end{equation*}
From the Petree-Tartar lemma~\cite[Lemma A.38 page 469]{Ern}, it is standard to verify that $A$ is self-adjoint and accretive on $\mathcal{V}(\Omega)$. Therefore, following for instance~\cite[Chapter~7]{Evans}, there exists a unique $v\in \L^2(0,T;\mathcal{V}(\Omega)) \cap \H^1(0,T; \mathcal{V}'(\Omega))$ such that
\begin{equation*}
\langle \dot{v}; \varphi\rangle_{\mathcal{V}(\Omega)';\mathcal{V}(\Omega)} + 
\kappa\int_{\Omega} \nabla v : \nabla \varphi \, \d \Omega = 
\langle F; \varphi\rangle_{\mathcal{V}(\Omega)';\mathcal{V}(\Omega)}
+ \langle G; \varphi\rangle_{\mathcal{V}(\Gamma_N)';\mathcal{V}(\Gamma_N)}
\end{equation*}
for all $\varphi \in \mathcal{V}(\Omega)$, almost everywhere in $(0,T)$. After integration by parts, we obtain
\begin{equation*}
\dot{v} -\kappa \Delta v - F = 0  \quad \text{in } \mathcal{V}(\Omega)', 
\qquad \text{and} \qquad
\kappa \frac{\p v}{\p n} - G = 0 \quad \text{in } \mathcal{V}(\Gamma_N)'.
\end{equation*}
Then, from Lemma~\ref{lemma-const}, there exists $\mathfrak{p}:t\mapsto \mathfrak{p}(t)\in \R$ such that $\displaystyle \kappa \frac{\p v}{\p n} +\mathfrak{p}\, n = G$ almost everywhere in (0,T). Next, by deriving in time~\eqref{sys-P-const}, we deduce $\displaystyle \int_{\Gamma_N}\langle \dot{v}; n\rangle_{\HH^{-1/2}(\Gamma_N);\HH^{1/2}(\Gamma_N)}\, \d \Gamma_N = 0$, and since $F\in \L^2(0, T; \LL^2(\Omega))$, $G\in \L^2(0,T;\HH^{1/2}(\Gamma_N))$ and $v_0\in \WW^{2/{p'},p}(\Omega) \hookrightarrow \WW^{1,p}(\Omega) \hookrightarrow \HH^1(\Omega)$, by taking the scalar product of~\eqref{sys-P-1} by $\dot{v}$ and integrating by parts, we obtain
\begin{equation*}
\|\dot{v}\|^2_{\L^2(0,T;\LL^2(\Omega))} \leq
\frac{\kappa}{2}\|\nabla v_0\|^2_{\LL^2(\Omega)} +
\|F\|_{\L^2(0,T;\LL^2(\Omega))} \|\dot{v}\|_{\L^2(0,T;\LL^2(\Omega))}
+ \|G\|_{\L^2(0,T;\HH^{1/2}(\Gamma_N))} \|\dot{v}\|_{\L^2(0,T;\LL^2(\Omega))}.
\end{equation*}
The Young's inequality then shows that $\dot{v} \in \L^2(0,T;\LL^2(\Omega))$, and $-\kappa \Delta v = F-\dot{v}$ too. Therefore $v\in \L^2(0,T;\HH^2(\Omega)) \cap \H^1(0,T;\LL^2(\Omega))$. Thus the operator $A$ defined above enjoys the properties of the $L^p$-maximal regularity for $p=2$, and so for any $p>3$, that is
\begin{equation*}
\begin{array} {rcl}
& & v\in \L^p(0,T;\HH^2(\Omega)) \cap \W^{1,p}(0,T;\LL^2(\Omega)), \\
& & 
\| v\|_{\L^p(0,T;\HH^2(\Omega)) \cap \W^{1,p}(0,T;\LL^2(\Omega))} \leq 
C(T)\left(\|F\|_{\L^p(0,T;\LL^2(\Omega))}
+ \|G\|_{\L^p(0,T;\HH^{1/2}(\Gamma_N)) \cap \W^{1/(2p'),p}(0,T;\LL^2(\Gamma_N))}
 \right),
\end{array}
\end{equation*}
where the constant~$C(T)>0$ is, as previously mentioned, non-decreasing with respect to~$T$. Further, estimate~\eqref{est-trace-emb2} shows that $\displaystyle \frac{\p v}{\p n} \in \W^{1/(2p'),p}(0,T;\HH^{1/2-1/p}(\Gamma_N)')$, and enables us to obtain
\begin{equation}
\left\|\frac{\p v}{\p n} \right\|_{\W^{1/(2p'),p}(0,T;\HH^{1/2-1/p}(\Gamma_N)')} \leq 
C(T)\left(\|F\|_{\L^p(0,T;\LL^2(\Omega))}
+ \|G\|_{\L^p(0,T;\HH^{1/2}(\Gamma_N)) \cap \W^{1/(2p'),p}(0,T;\LL^2(\Gamma_N))}
 \right).
\label{est-normder}
\end{equation}
Further, we deduce from Lemma~\ref{lemma-DiNezza} with~$\gamma = 1/2-1/2p$, that is $\gamma-1/p = (1-3/p)/2 >0$, the following estimate
\begin{equation*}
\displaystyle
\left\|\frac{\p v}{\p n} \right\|_{\L^{\infty}(0,T;\HH^{1/2-1/p}(\Gamma_N)')} 
\leq
\displaystyle
\left\|\frac{\p v_0}{\p n} \right\|_{\HH^{1/2-1/p}(\Gamma_N)'} +
\displaystyle CT^{(1-3/p)/2}
\left\|\frac{\p v}{\p n} \right\|_{\W^{1/(2p'),p}(0,T;\HH^{1/2-1/p}(\Gamma_N)')}.
\end{equation*}
Since $v_0 \in \W^{2/p',p}(\Omega)$, we have $\displaystyle \frac{\p v_0}{\p n} \in \W^{1-3/p,p}(\Gamma_N) \hookrightarrow \LL^p(\Gamma_N) \hookrightarrow \HH^{1/2-1/p}(\Gamma_N)'$. Therefore we deduce
\begin{equation*}
\displaystyle
\left\|\frac{\p v}{\p n} \right\|_{\L^{\infty}(0,T;\HH^{1/2-1/p}(\Gamma_N)')} \leq 
C(T) \left(
\|v_0 \|_{\WW^{2/p',p}(\Omega)} +
\left\|\frac{\p v}{\p n} \right\|_{\W^{1/(2p'),p}(0,T;\HH^{1/2-1/p}(\Gamma_N)')}
\right)
\end{equation*}
which, combined with~\eqref{est-normder}, yields
\begin{eqnarray}
\displaystyle
\left\|\frac{\p v}{\p n} \right\|_{\L^{\infty}(0,T;\HH^{1/2-1/p}(\Gamma_N)')}
& \leq & C(T) \left(
\|v_0 \|_{\WW^{2/p',p}(\Omega)} +
\|F\|_{\L^p(0,T;\LL^2(\Omega))} \right. \nonumber \\
& & \left. 
+ \|G\|_{\L^p(0,T;\HH^{1/2}(\Gamma_N)) \cap \W^{1/(2p'),p}(0,T;\LL^2(\Gamma_N))}
\right).
\label{est-normder2}
\end{eqnarray}
Then from~\eqref{sys-P-2}, since $n \in \WW^{2-1/p,p}(\Gamma_N) \hookrightarrow \HH^{2-1/p}(\Gamma_N) \hookrightarrow \HH^{1/2-1/p}(\Gamma_N)$, we can deduce
\begin{equation*}
\begin{array} {l}
\mathfrak{p} = \displaystyle
 \frac{1}{|\Gamma_N|}\left(
\int_{\Gamma_N} G \cdot n \, \d \Gamma_N 
-\kappa\left\langle\frac{\p v}{\p n}; n\right\rangle_{\HH^{1/2-1/p}(\Gamma_N)',\HH^{1/2-1/p}(\Gamma_N)}
\right)
 \in \W^{1/2-1/(2p),p}(0,T;\R), 
\nonumber \\
\displaystyle
\| \mathfrak{p}\|_{\W^{1/(2p'),p}(0,T;\R)} \leq C\left(
\|G\|_{\W^{1/(2p'),p}(0,T;\LL^2(\Gamma_N))} + 
\left\|\frac{\p v}{\p n} \right\|_{\W^{1/(2p'),p}(0,T;\HH^{1/2-1/p}(\Gamma_N)')}
\right), \\
\| \mathfrak{p}\|_{\L^{\infty}(0,T;\R)} \leq C\left(
\displaystyle
\|G\|_{\L^{\infty}(0,T;\LL^2(\Gamma_N))} + 
\left\|\frac{\p v}{\p n} \right\|_{\L^{\infty}(0,T;\HH^{1/2-1/p}(\Gamma_N)')}
\right).
\end{array}
\end{equation*}
Combined with~\eqref{est-normder} and~\eqref{est-normder2}, these estimates yield
\begin{eqnarray}
\| \mathfrak{p}\|_{\mathcal{P}_{p,T}} & \leq & C(T)\left(
\|v_0\|_{\WW^{2/p',p}(\Omega)}+
\|F\|_{\L^p(0,T;\LL^2(\Omega))} \right. \nonumber \\
& & \left.+ \|G\|_{\L^p(0,T;\HH^{1/2}(\Gamma_N)) \cap \W^{1/(2p'),p}(0,T;\LL^2(\Gamma_N))}
 + \|G\|_{\L^{\infty}(0,T;\L^{2}(\Gamma_N))}
\right) \nonumber \\
& \leq & C(T)
\left(
\|v_0\|_{\WW^{2/p',p}(\Omega)} +
\|F\|_{\mathcal{F}_{p,T}(\Omega)} + \|G\|_{\mathcal{G}_{p,T}(\Gamma_N)}
 \right)
.\label{est-pression}
\end{eqnarray}
Then $\displaystyle \kappa \frac{\p v}{\p n} = G -\mathfrak{p}\, n \in \mathcal{G}_{p,T}(\Gamma_N)$ (since~$\mathfrak{p}$ does not depend on the space variable), and the existence of~$v$ in~$\dot{\mathcal{U}}_{p,T}(\Omega)$ follows from Proposition~\ref{prop-Pruess}. In particular, Lemma~\ref{lemma-DiNezza} with~$\gamma = 1-1/(2p)$, that is $\gamma-1/p = 1-3/(2p)>0$ enables us to obtain
\begin{eqnarray}
\|v\|_{\L^{\infty}(0,T;\LL^p(\Gamma_N))} & \leq & 
\|v_0\|_{\LL^p(\Gamma_N)} + 
CT^{1-3/(2p)}\|v\|_{\W^{1-1/(2p)}(0,T;\LL^p(\Gamma_N))}\nonumber \\
& \leq & C(T)\left(
\|v_0\|_{\W^{2/p',p}(\Omega)} +
\|v\|_{\dot{\mathcal{U}}_{p,T}(\Omega)}\right),
\label{est-Lp}
\end{eqnarray}
where we have used~\eqref{est-trace-emb}. Further, estimate~\eqref{est-Pruess} yields
\begin{equation*}
\begin{array} {rcl}
\|v\|_{\dot{\mathcal{U}}_{p,T}(\Omega)} & \leq &  C(T) \left(
\|v_0\|_{\mathbf{W}^{2/{p'},p}(\Omega)} + 
\|F\|_{\mathcal{F}_{p,T}(\Omega)} + 
\|G\|_{\mathcal{G}_{p,T}(\Gamma_N)} +
\|\mathfrak{p}\, n\|_{\mathcal{G}_{p,T}(\Gamma_N)}
\right) \\
& \leq &
 C(T) \left(
\|v_0\|_{\mathbf{W}^{2/{p'},p}(\Omega)} + 
\|F\|_{\mathcal{F}_{p,T}(\Omega)} + 
\|G\|_{\mathcal{G}_{p,T}(\Gamma_N)} +
\|\mathfrak{p}\|_{\mathcal{P}_{p,T}}
\right)
\end{array}
\end{equation*}
which, combined with~\eqref{est-pression} and~\eqref{est-Lp}, leads us to estimate~\eqref{est-Pruess2} and completes the proof.
\end{proof}

Note that the constraint~\eqref{sys-P-const} is the linearization of the third equation of~\eqref{sysmain3}. Therefore, we will need to consider~\eqref{sys-P-const} with a non-homogeneous right-hand-side.

\begin{corollary} \label{coro-linNH}
Given
\begin{equation*}
F \in \mathcal{F}_{p,T}(\Omega), \quad 
G \in \mathcal{G}_{p,T}(\Gamma_N), \quad
H \in \W^{1-1/{2p},p}(0,T;\R), \quad
v_0 \in \mathbf{W}^{2/{p'},p}(\Omega) \cap \mathbf{W}^{1/{p'},p}_{0,D}(\Omega),
\end{equation*}
with the compatibility condition $\displaystyle \kappa\frac{\p v_0}{\p n} = G(\cdot,0)$ on $\Gamma_N$ and $\displaystyle \int_{\Gamma_N} v_0\cdot n \, \d \Gamma_N = H(0)$, there exists a unique couple $(v,\mathfrak{p}) \in \dot{\mathcal{U}}_{p,T}(\Omega) \times \mathcal{P}_{p,T}$ solution to the following system:
\begin{subequations}\label{sys-Pruess-constraint-NH}
\begin{eqnarray}
\displaystyle \dot{v} - \kappa \Delta v = F & & \text{in } \Omega \times (0,T), \\
\displaystyle \kappa \frac{\p v}{\p n} +\mathfrak{p}\, n = G & & \text{on } \Gamma_N \times (0,T), \\
\displaystyle \int_{\Gamma_N} v\cdot n\, \d \Gamma_N = H & & \text{in } (0,T), 
\label{sys-P-const-NH}\\
v = 0 & & \text{on } \Gamma_D \times (0,T), \\
v(\cdot,0) = v_0 & & \text{in } \Omega.
\end{eqnarray}
\end{subequations}
It satisfies
\begin{equation}
\|v\|_{\dot{\mathcal{U}}_{p,T}(\Omega)} 
+ \| v\|_{\L^{\infty}(0,T;\LL^p(\Gamma_N))}
+ \| \mathfrak{p}\|_{\mathcal{P}_{p,T}} 
 \leq  C_0(T) \left(
\|v_0\|_{\mathbf{W}^{2/{p'},p}(\Omega)} + 
\|F\|_{\mathcal{F}_{p,T}(\Omega)} + 
\|G\|_{\mathcal{G}_{p,T}(\Gamma_N)} 
+ \|H\|_{\mathcal{H}_{p,T}}\right)
\label{est-Pruess3}
\end{equation}
where the constant $C_0(T)>0$ is non-decreasing with respect to $T$.
\end{corollary}

\begin{proof}
Note that $H(0) \in \R$ does not depend on the space variable. Consider any extension~$\overline{H}_0 \in \WW^{2/{p'},p}(\Omega)$ of $\displaystyle  \frac{1}{|\Gamma_N|} H(0)n$ such that
\begin{equation*}
{\overline{H}_0}_{| \Gamma_N} = \frac{1}{|\Gamma_N|} H(0)n, \qquad
\| \overline{H}_0 \|_{\WW^{2/{p'},p}(\Omega)} 
\leq C\|H(0) \|_{\R}.
\end{equation*}
This is possible when $n \in \WW^{2/{p'}-1/p,p}(\Gamma_N) = \WW^{2-3/p,p}(\Gamma_N)$, which is the case because we assumed $n\in \WW^{2-1/p,p}(\Gamma_N)$ in section~\ref{sec-assgeo}. From $H$ we define an extension $\bar{H}$ by solving the following heat equation:
\begin{equation*}
\begin{array} {rcl}
\dot{\bar{H}}- \kappa\Delta \bar{H} = 0  & & \text{in } \Omega \times (0,T), \\
\bar{H} = \displaystyle \frac{1}{|\Gamma_N|}H n  
& &  \text{on } \Gamma_N \times (0,T), \\
\bar{H} = 0  & & \text{on } \Gamma_D \times (0,T), \\
\bar{H}(0) = \overline{H}_0 & &  \text{in } \Omega.
\end{array}
\end{equation*}
Since $H$ does not depend on the space variable and $n\in \W^{2-1/p,p}(\Gamma_N)$, we see easily that the Dirichlet condition on~$\Gamma_N$ satisfies
\begin{equation*}
\displaystyle \frac{1}{|\Gamma_N|}H n \in \W^{1-1/(2p),p}(0,T;\WW^{2-1/p,p}(\Gamma_N)) \hookrightarrow 
\W^{1-1/(2p),p}(0,T;\LL^p(\Gamma_N))\cap\L^p(0,T;\WW^{2-1/p,p}(\Gamma_N)),
\end{equation*}
and since the compatibility condition ${\overline{H}_0}_{|\Gamma_N} = \displaystyle \frac{1}{|\Gamma_N|} H(0)n$ is satisfied, we derive from~\cite{Pruess2002} the following estimate
\begin{equation}
\|\bar{H}(0)\|_{\mathbf{W}^{2/{p'},p}(\Omega)}
+\kappa\left\|\frac{\p \bar{H}}{\p n}\right\|_{\mathcal{G}_{p,T}(\Gamma_N)} \leq
C\left( \| H \|_{\W^{1-1/{2p},p}(0,T;\R)} +
\|H(0)\|_{\R}
\right) \leq C \|H \|_{\mathcal{H}_{p,T}}.
\label{est-dirichlet-hbar}
\end{equation}
Now define $\bar{v} = v - \bar{H}$, which satisfies
\begin{equation*}
\begin{array} {rcl}
\displaystyle \dot{\bar{v}} - \kappa \Delta \bar{v} = F & & \text{in } \Omega \times (0,T), \\[5pt]
\displaystyle \kappa \frac{\p \bar{v}}{\p n} + \mathfrak{p}\, n = 
G - \kappa\frac{\p \bar{H}}{\p n} & & \text{on } \Gamma_N \times (0,T), \\[5pt]
\displaystyle \int_{\Gamma_N} \bar{v}\cdot n\, \d \Gamma_N = 0 & & \text{in } (0,T), 
\\
\bar{v} = 0 & & \text{on } \Gamma_D \times (0,T), \\
\bar{v}(\cdot,0) = v_0 - \bar{H}(0) & & \text{in } \Omega.
\end{array}
\end{equation*}
From Proposition~\ref{prop-pressure}, there exists a unique couple $(\bar{v},\mathfrak{p})\in \dot{\mathcal{U}}_{p,T}(\Omega) \times \mathcal{P}_{p,T}$ solution of the system above, and satisfying
\begin{equation*}
\begin{array} {rcl}
\|\bar{v}\|_{\dot{\mathcal{U}}_{p,T}(\Omega)} + 
\| \mathfrak{p} \|_{\mathcal{C}([0,T];\R)} 
& \leq & C \left(
\|v_0\|_{\mathbf{W}^{2/{p'},p}(\Omega)} + 
\|F\|_{\mathcal{F}_{p,T}(\Omega)} + 
\|G\|_{\mathcal{G}_{p,T}(\Gamma_N)} \right.\\
& & \left. + \displaystyle \|\bar{H}(0)\|_{\mathbf{W}^{2/{p'},p}(\Omega)}
+\kappa\left\|\frac{\p \bar{H}}{\p n}\right\|_{\mathcal{G}_{p,T}(\Gamma_N)}
\right).
\end{array}
\end{equation*}
Combining this estimate with~\eqref{est-dirichlet-hbar} enables us to conclude the proof.
\end{proof}

Estimate~\eqref{est-Pruess3} of Corollary~\ref{coro-linNH} is used for the fixed-point strategy in section~\ref{sec-well-main}.

\section{Local-in-time wellposedness for the main system} \label{sec-well-main} \label{sec-well-1}

System~\eqref{sysmain3} can be rewritten in the following form
\begin{equation*}
\begin{array} {rcl}
\displaystyle \dot{v} - \kappa \Delta v = f+ F(v) & & \text{in } \Omega \times (0,T), \\
\displaystyle \kappa \frac{\p v}{\p n} + \mathfrak{p}\, n = g+G(v) & & \text{on } \Gamma_N \times (0,T), \\[5pt]
\displaystyle \int_{\Gamma_N} v\cdot n\, \d \Gamma_N = H(v) & & \text{in } (0,T), 
\label{sys-P-const-NH2}\\
v = 0 & & \text{on } \Gamma_D \times (0,T), \\
v(\cdot,0) = \udotini & & \text{in } \Omega,
\end{array}
\end{equation*}
where we introduce
\begin{subequations} \label{NL-terms}
\begin{eqnarray}
u(\cdot,t) & = &  \int_0^t v(\cdot,s) \d s, \label{NL0} \\
F(v) & = & \divg\left( \sigma(\nabla u)\right), \label{NL1}\\
G(v,\mathfrak{p}) & = & -\sigma(\nabla u)n + \mathfrak{p}\left(\I - \mathrm{cof}(\Phi(u))\right)n, 
\label{NL2}\\
H(v) & = & \int_{\Gamma_N} v\cdot\left(\I - \cof(\Phi(u) )\right)n\, \d \Gamma_N.
\label{NL3}
\end{eqnarray}
\end{subequations}
Solutions of~\eqref{sysmain3} are fixed points of the following mapping
\begin{equation}
\begin{array} {rccl}
\mathcal{K} : & \dot{\mathcal{U}}_{p,T}(\Omega) \times  \mathcal{P}_{p,T} &
\rightarrow & \dot{\mathcal{U}}_{p,T}(\Omega) \times  \mathcal{P}_{p,T} \\
& (v_a,\mathfrak{p}_a) & \mapsto & (v_b,\mathfrak{p}_b)
\end{array}
\label{def-contraction}
\end{equation}
where $(v_b,\mathfrak{p}_b)$ is the solution of system~\eqref{sys-Pruess-constraint-NH} with $(F,G,H,v_0)$ replaced by $(f+F(v_a), g+G(v_a,\mathfrak{p}_a), H(v_a), \udotini)$ as data. Following Corollary~\ref{coro-linNH}, these data must have the required regularity, which is proven in Proposition~\ref{prop-estxlip}. Further, to confirm that $\mathcal{K}$ is well-defined, the data must satisfied compatibility conditions, namely
\begin{equation*}
\kappa \frac{\p \dot{u}_0}{\p n} = g(\cdot,0) + G(v_a,\mathfrak{p}_a)(\cdot,0), \qquad
\int_{\Gamma_N} \dot{u}_0\cdot n\, \d \Gamma_N = H(v_a)(0).
\end{equation*}
Since we assume that $u_0 = 0$, $\displaystyle \kappa \frac{\p \dot{u}_0}{\p n} + \sigma(0)n = g(\cdot,0)$ on~$\Gamma_N$ and $\displaystyle \int_{\Gamma_N} \dot{u}_0 \cdot n\, \d \Gamma_N = 0$, these conditions are automatically satisfied.
Now for $R>0$ define the set
\begin{equation}
\mathcal{B}_R(T) = \left\{ (v,\mathfrak{p}) \in
\dot{\mathcal{U}}_{p,T}(\Omega) \times  \mathcal{C}([0,T];\R) \mid 
\| v\|_{\dot{\mathcal{U}}_{p,T}(\Omega)} 
+ \| v_{|\Gamma_N}\|_{\L^{\infty}(0,T;\LL^p(\Gamma_N))}
+ \|\mathfrak{p}\|_{\mathcal{P}_{p,T}}
\leq 2C_0(T)R
\right\},
\label{def-ball}
\end{equation}
where the constant $C_0(T)$ is the one which appears in estimate~\eqref{est-Pruess3}. 
The ball~$\mathcal{B}_R(T)$ is clearly closed in $\dot{\mathcal{U}}_{p,T}(\Omega) \times  \mathcal{P}_{p,T} $. Let us prove that~$\mathcal{K}$ is a contraction in $\mathcal{B}_R(T)$, for $R$ large enough and $T$ small enough. For that we need Lipschitz estimates on the nonlinear terms~\eqref{NL1}--\eqref{NL3}. We first prove a set of technical lemmas.

\subsection{Technical lemmas}
\begin{lemma}
Let $B$ be a Banach space and $\varphi \in \W^{1,p}(0,T;B)$. We have
\begin{eqnarray}
\|\varphi \|_{\L^{\infty}(0,T;B)} & \leq & 
\|\varphi(0)\|_B + T^{1/{p'}} \| \dot{\varphi} \|_{\L^p(0,T;B)}, \label{basicest1}\\
\|\varphi \|_{\L^p(0,T;B)} & \leq & 
T^{1/p}\|\varphi(0)\|_B + T \| \dot{\varphi} \|_{\L^p(0,T;B)}, \label{basicest2} \\
\|\varphi\|_{\W^{\alpha,p}(0,T;B)} & \leq &  
T^{(1-\alpha)/p}\left(\|\varphi(0)\|_B +
 \|\varphi\|_{\W^{1,p}(0,T;B)} \right),
\label{basicest3}\\
\|\varphi\|_{\W^{1,p}(0,T;B)} & \leq & T^{1/p}\|\varphi(0)\|_B +C\| \dot{\varphi} \|_{\L^p(0,T;B)},
\label{basicest4}
\end{eqnarray}
for all $\alpha \in (0,1)$, assuming $T\leq 1$.
\end{lemma}

Since the local-in-time existence result, namely Theorem~\ref{th-locexist}, is obtained by assuming~$T$ small enough, in the rest of this section we will assume~$T\leq 1$, for the sake of concision in the different estimates, but without loss of generality.

\begin{proof}
We write $\varphi(t) = \varphi(0) + \displaystyle\int_0^t \dot{\varphi}(s) \d s$, and from the H\"{o}lder's inequality we get
\begin{equation*}
\|\varphi(t) \|_B \leq \|\varphi(0)\|_B + 
t^{1/{p'}}\|\dot{\varphi} \|_{\L^p(0,T;B)},
\end{equation*}
which leads to the first estimate. Since
\begin{equation*}
\|\varphi\|_{\L^p(0,T;B)} \leq T^{1/p} \|\varphi \|_{\L^{\infty}(0,T;B)},
\end{equation*}
we deduce the second estimate from the first one. Estimate~\eqref{basicest3} is deduced by interpolation:
\begin{equation*}
\|\varphi\|_{\W^{\alpha,p}(0,T;B)} \leq 
\|\varphi\|_{\L^{p}(0,T;B)}^{1-\alpha} \|\varphi\|^{\alpha}_{\W^{1,p}(0,T;B)}.
\end{equation*}
Using~\eqref{basicest2}, the subadditivity of the function $x\mapsto x^{1-\alpha}$ and the Young's inequality, this yields
\begin{equation*}
\begin{array} {rcl}
\|\varphi\|_{\W^{\alpha,p}(0,T;B)} & \leq & 
\left( T^{1/p}\|\varphi(0)\|_B + T \| \dot{\varphi} \|_{\L^p(0,T;B)} \right)^{1-\alpha}
\|\varphi\|^{\alpha}_{\W^{1,p}(0,T;B)} \\
& \leq & \left(T^{(1-\alpha)/p}\|\varphi(0)\|^{1-\alpha}_B 
+ T^{1-\alpha} \| \dot{\varphi} \|^{1-\alpha}_{\L^{p}(0,T;B)} \right)\|\varphi\|^{\alpha}_{\W^{1,p}(0,T;B)} \\
& \leq & T^{(1-\alpha)/p}\|\varphi(0)\|^{1-\alpha}_B \|\varphi\|^{\alpha}_{\W^{1,p}(0,T;B)}
+ T^{1-\alpha} \|\varphi\|_{\W^{1,p}(0,T;B)} \\
& \leq &  T^{(1-\alpha)/p} 
\left( (1-\alpha)\|\varphi(0)\|_B + \alpha \|\varphi\|_{\W^{1,p}(0,T;B)}\right)
+ T^{1-\alpha} \|\varphi\|_{\W^{1,p}(0,T;B)},
\end{array}
\end{equation*}
leading to the announced estimate, as $T^{1-\alpha} \leq T^{(1-\alpha)/p}$ when we assume $T\leq 1$. Finally, estimate~\eqref{basicest4} is obtained as follows
\begin{equation*}
\begin{array} {rcl}
\| \varphi \|_{\W^{1,p}(0,T;B)}
& \leq & \| \varphi \|_{\L^{p}(0,T;B)} + \| \dot{\varphi} \|_{\L^{p}(0,T;B)} \\
& \leq & T^{1/p}\|\varphi(0)\|_B + (T +1)\| \dot{\varphi} \|_{\L^p(0,T;B)} \\
& \leq & T^{1/p}\|\varphi(0)\|_B +C\| \dot{\varphi} \|_{\L^p(0,T;B)} ,
\end{array}
\end{equation*}
where we have used~\eqref{basicest2}, concluding the proof.
\end{proof}

We will also need a result concerning the stability by product of fractional Sobolev spaces. We deduce from~\cite[Lemma~4.1]{Brezis2001}, the so-called Runst-Sickel lemma, a consequence of~\cite[p.~345]{Runst1996}, the following result:

\begin{lemma} \label{lemma-Brezis}
Assume $\varphi_1 \in \W^{\beta,p}(0,T;\R)$ for some $1/p<\beta<1$, and $\varphi_2 \in \W^{1,p}(0,T;\R)$ such that $\varphi_2(0) = 0$. Then we have
\begin{equation*}
\| \varphi_1 \varphi_2 \|_{\W^{\beta,p}(0,T;\R)} \leq C T^{(1-\beta)/p} \left(
\| \varphi_1\|_{\L^{\infty}(0,T;\R)} + \| \varphi_1 \|_{\W^{\beta,p}(0,T;\R)}
\right)
\| \varphi_2 \|_{\W^{1,p}(0,T;\R)}
\end{equation*}
where the constant $C>0$ is independent of~$T\leq 1$.
\end{lemma}

\begin{proof}
When $\beta p >1$ we have $\W^{\beta,p}(0,T;\R) \hookrightarrow \mathcal{C}([0,T];\R)$, and $\varphi_1(0)$, $\varphi_2(0)$ make sense. Further,~\cite[Lemma~4.1]{Brezis2001} yields
\begin{equation*}
\| \varphi_1 \varphi_2 \|_{\W^{\beta,p}(0,T;\R)} \leq C 
\left(\|  \varphi_1 \|_{\L^{\infty}(0,T;\R)}
\| \varphi_2 \|_{\W^{\beta,p}(0,T;\R)}
 + \| \varphi_1 \|_{\W^{\beta,p}(0,T;\R)}
\|  \varphi_2 \|_{\L^{\infty}(0,T;\R)} \right).
\end{equation*}
By following the proof of~\cite[Lemma~4.1]{Brezis2001}, one can show that the constant $C>0$ depends only on $\beta$ and $p$, but not on the size of the domain~$T$. Combined with~\eqref{basicest1}, we deduce
\begin{equation*}
\| \varphi_1 \varphi_2 \|_{\W^{\beta,p}(0,T;\R)} \leq C \left(
\|  \varphi_1 \|_{\L^{\infty}(0,T;\R)}
\| \varphi_2 \|_{\W^{\beta,p}(0,T;\R)}
 + T^{1/{p'}}\| \varphi_1 \|_{\W^{\beta,p}(0,T;\R)}
\| \dot{\varphi}_2 \|_{\L^p(0,T;\R)}\right),
\end{equation*}
because we have assumed $\varphi_2(0) = 0$. Next we use~\eqref{basicest3} for controlling $\| \varphi_2 \|_{\W^{\beta,p}(0,T;\R)}$ as follows
\begin{equation*}
\| \varphi_2 \|_{\W^{\beta,p}(0,T;\R)} \leq T^{(1-\beta)/p} 
\| \varphi_2 \|_{\W^{1,p}(0,T;\R)},
\end{equation*}
leading to 
\begin{equation*}
\begin{array} {rcl}
\| \varphi_1 \varphi_2 \|_{\W^{\beta,p}(0,T;\R)} & \leq & C \left(
T^{(1-\beta)/p}\|  \varphi_1 \|_{\L^{\infty}(0,T;\R)}
\| \varphi_2 \|_{\W^{1,p}(0,T;\R)} \right. \\
& &  \left. +
T^{1/{p'}}\| \varphi_1 \|_{\W^{\beta,p}(0,T;\R)}
\| \varphi_2 \|_{\W^{1,p}(0,T;\R)}\right).
\end{array}
\end{equation*}
Since $p\geq 1$ and $\beta>1/p$, we have $(1-\beta)/p < 1/p'$, and therefore for $T\leq 1$ we have $T^{(1-\beta)/p} \geq T^{1/{p'}}$. Thus the announced estimate follows.
\end{proof}

\subsection{Lipschitz estimates} \label{sec-Lip}

We derive the following intermediate lemmas before stating Lipschitz estimates in Proposition~\ref{prop-estxlip}.

\begin{lemma} \label{lemma-sigmax}
Under Assumption~$\mathbf{A1}$, if $v_1, v_2 \in \mathcal{B}_R(T)$ define $u_i(\cdot,t) =  \displaystyle \int_0^t v_i(\cdot,s)\d s$ for $i\in \{1,2\}$, then there exists a constant $C_R(T)$, non-decreasing with respect to~$R$, and non -decreasing with respect to~$T$, such that
\begin{equation}
\| \sigma(\nabla u_1) - \sigma(\nabla u_2) \|_{\L^p(0,T;\WWW^{1,p}(\Omega))} \leq 
C_R(T) T \|\nabla v_1 - \nabla v_2 \|_{\L^p(0,T;\WWW^{1,p}(\Omega))}.
\label{estlip-sigma}
\end{equation}
In particular, for all $v\in \mathcal{B}_R(T)$, if $u(\cdot,t) =  \displaystyle\int_0^t v(\cdot,s)\d s$, then
\begin{equation}
\| \sigma(\nabla u) \|_{\L^p(0,T;\WWW^{1,p}(\Omega))} \leq 
T^{1/p}\| \sigma(0) \|_{\WWW^{1,p}(\Omega)} + 
C_R(T) T \|\nabla v \|_{\L^p(0,T;\WWW^{1,p}(\Omega))} .
\label{estlip-sigma0}
\end{equation}
\end{lemma}

\begin{proof}
From the mean value theorem, we have
\begin{equation*}
\| \sigma(\nabla u_1) - \sigma(\nabla u_2) \|_{\mathbb{W}^{1,p}(\Omega)} 
 \leq 
\displaystyle\sup_{s \in [0,1]} \left(
\left\|\sigma'(s \nabla u_1 + (1-s)\nabla u_2) \right\|_{\mathscr{L}\left(\mathbb{W}^{1,p}(\Omega);\mathbb{W}^{1,p}(\Omega)\right)}
\right) \|\nabla u_1 - \nabla u_2\|_{\mathbb{W}^{1,p}(\Omega)}, 
\end{equation*}
where~$\sigma'$ is introduced in~\eqref{sigmaprime}. Therefore, using that the set $\mathcal{B}_R(T)$ is convex, we get
\begin{equation*}
\| \sigma(\nabla u_1) - \sigma(\nabla u_2) \|_{\L^p(0,T;\mathbb{W}^{1,p}(\Omega))} 
 \leq 
C_R(T) \|\nabla u_1 - \nabla u_2\|_{\L^p(0,T;\mathbb{W}^{1,p}(\Omega))}
\end{equation*} 
where we set
\begin{equation*}
C_R(T):=\sup_{
\| \nabla \hat{v}\|_{\L^{p}(0,T;\WWW^{1,p}(\Omega))} \leq 2C_0(T)R
} 
\left(
\left\|\sigma'(\nabla\hat{u})\right\|_{\L^{\infty}(0,T;\mathscr{L}\left(\mathbb{W}^{1,p}(\Omega);\mathbb{W}^{1,p}(\Omega)\right))}\right).
\end{equation*}
where we denote~$\hat{v}$ such that $\hat{u}(\cdot,t) = \displaystyle \int_0^t \hat{v}(\cdot,s)\d s$. Since from~\eqref{basicest2} we have
\begin{equation*}
\| \nabla \hat{u}\|_{\L^{\infty}(0,T;\WWW^{1,p}(\Omega))} \leq T \|\nabla \hat{v}\|_{\L^p(0,T;\WWW^{1,p}(\Omega))} \leq 2C_0(T)TR,
\end{equation*}
from Assumption~$\mathbf{A1}$ the constant~$C_R(T)$ is well-defined. Note that~$C_{R}(T)$ is non-decreasing with respect to~$R$, and also non-decreasing with respect to~$T$, as $C_0(T)$ is non-decreasing with respect to~$T$. We then obtain~\eqref{estlip-sigma} by using~\eqref{basicest2} with $\varphi = \nabla u_1 - \nabla u_2$ that satisfies $\varphi(0) = 0$. From~\eqref{estlip-sigma}, we get~\eqref{estlip-sigma0} by choosing $v_1 = v$ and $v_2 = 0$, which concludes the proof.
\end{proof}

\begin{lemma} \label{lemma-cof}
Assume that $u\in \W^{1,p}(0,T; \WW^{2,p}(\Omega))$satisfies $u(\cdot,0) \equiv 0$ and recall the notation $\Phi(u) = \I+\nabla u$. Then
\begin{equation} \label{est-cof1}
\begin{array} {rcl}
\| \I -\cof \Phi(u) \|_{\W^{1,p}(0,T;\WWW^{1,p}(\Omega))} & \leq &
C
\left(\|\nabla \dot{u}\|_{\L^{p}(0,T;\WWW^{1,p}(\Omega)) } \right)
\left(1 + 
\|\nabla \dot{u}\|^{d-2}_{\L^{p}(0,T;\WWW^{1,p}(\Omega)) }
\right).
\end{array}
\end{equation}
Furthermore, if $u_1, u_2\in \W^{1,p}(0,T; \WW^{2,p}(\Omega))$ such that $u_1(0) = u_2(0) = u_0 = 0$, then
\begin{equation} \label{est-cof12}
\begin{array} {rcl}
\| \cof \Phi(u_1) -\cof \Phi(u_2) \|_{\W^{1,p}(0,T;\WWW^{1,p}(\Omega))} & \leq &
C
\|\nabla \dot{u}_1 - \nabla \dot{u}_2\|_{\L^{p}(0,T;\WWW^{1,p}(\Omega)) }
\\
& & \times \left(1 + 
\|\nabla \dot{u}_1\|^{d-2}_{\L^{p}(0,T;\WWW^{1,p}(\Omega)) }
+ \|\nabla \dot{u}_2\|^{d-2}_{\L^{p}(0,T;\WWW^{1,p}(\Omega)) }
\right).
\end{array}
\end{equation}
\end{lemma}

\begin{proof}
Let us directly prove~\eqref{est-cof12}, as it implies~\eqref{est-cof1} by choosing $u_1=u_0=0$ and $u_2=u$. 
First, consider two matrix fields $A$,~$B \in \WWW^{1,p}(\Omega)$, playing the role of~$\Phi(u_1)$ and~$\Phi(u_2)$, respectively. Recall that $A\mapsto \cof(A)$ is a polynomial form of degree $d-1$ of the coefficients of~$A$, and since the space ${\mathbb{W}^{1,p}(\Omega)}$ is stable by product, following the estimate~\eqref{W-algebra}, we obtain the two following estimates
\begin{equation*}
\begin{array} {rcl}
\| \cof (A) -\cof (B) \|_{\WWW^{1,p}(\Omega)} & \leq & 
C \| A-B\|_{\WWW^{1,p}(\Omega)}\left(
1+ \|A\|^{d-2}_{\WWW^{1,p}(\Omega)} + \|B\|^{d-2}_{\WWW^{1,p}(\Omega)}
\right), \\
\displaystyle\left\|\frac{\p }{\p t} \left( \cof (A) -\cof (B) \right)\right\|_{\WWW^{1,p}(\Omega)} & \leq &
C \| \dot{A}-\dot{B}\|_{\WWW^{1,p}(\Omega)}\left(
1+ \|A\|^{d-2}_{\WWW^{1,p}(\Omega)} + \|B\|^{d-2}_{\WWW^{1,p}(\Omega)}
\right),
\end{array}
\end{equation*}
which yield
\begin{equation*}
\begin{array} {rcl}
\| \cof (A) -\cof (B) \|_{\L^p(0,T;\WWW^{1,p}(\Omega))} & \leq & 
C \| A-B\|_{\L^p(0,T;\WWW^{1,p}(\Omega))} \\
& & \times \left(
1+ \|A\|^{d-2}_{\L^{\infty}(0,T;\WWW^{1,p}(\Omega))} + \|B\|^{d-2}_{\L^{\infty}(0,T;\WWW^{1,p}(\Omega))}
\right), \\
\displaystyle\left\|\frac{\p }{\p t} \left( \cof (A) -\cof (B) \right)\right\|_{\L^p(0,T;\WWW^{1,p}(\Omega))} & \leq & 
C \| \dot{A}-\dot{B}\|_{\L^p(0,T;\WWW^{1,p}(\Omega))} \\
& &  \times\left(
1+ \|A\|^{d-2}_{\L^{\infty}(0,T;\WWW^{1,p}(\Omega))} + \|B\|^{d-2}_{\L^{\infty}(0,T;\WWW^{1,p}(\Omega))}
\right),
\end{array}
\end{equation*}
and thus
\begin{equation*}
\| \cof (A) -\cof (B) \|_{\W^{1,p}(0,T;\WWW^{1,p}(\Omega))}  \leq 
C \| A-B\|_{\W^{1,p}(0,T;\WWW^{1,p}(\Omega))}
\left(
1+ \|A\|^{d-2}_{\L^{\infty}(0,T;\WWW^{1,p}(\Omega))} + \|B\|^{d-2}_{\L^{\infty}(0,T;\WWW^{1,p}(\Omega))}
\right).
\end{equation*}
Next, we use estimate~\eqref{basicest1} for controlling in the right-hand-side the matrix fields $A$ and $B$ in $\L^{\infty}(0,T;\WWW^{1,p}(\Omega))$, as follows
\begin{equation*}
\begin{array} {rcl}
\|A\|_{\L^{\infty}(0,T;\WWW^{1,p}(\Omega))} & \leq &
\|A(0) \|_{\WWW^{1,p}(\Omega)} + T^{1/{p'}} 
\|\dot{A}\|_{\L^{p}(0,T;\WWW^{1,p}(\Omega))}\\
& \leq & C
+ T^{1/{p'}} 
\|\nabla\dot{u}\|_{\L^{p}(0,T;\WWW^{1,p}(\Omega))} ,
\end{array}
\end{equation*}
and further use~\eqref{basicest4} for controlling
\begin{equation*}
\| A-B\|_{\W^{1,p}(0,T;\WWW^{1,p}(\Omega))}  \leq 
\|A(0) - B(0) \|_{\WWW^{1,p}(\Omega)} + 
\|\dot{A}-\dot{B}\|_{\L^p(0,T;\WWW^{1,p}(\Omega))} 
 \leq  \|\dot{A}-\dot{B}\|_{\L^p(0,T;\WWW^{1,p}(\Omega))},
\end{equation*}
as $A(0) = B(0) = \I+\nabla u_0 = \I$, which leads us to the announced result.
\end{proof}

We deduce the Lipschitz properties for the nonlinear terms~\eqref{NL1}--\eqref{NL3}:

\begin{proposition} \label{prop-estxlip}
For all $(v,\mathfrak{p}) \in \mathcal{B}_R(T)$ we have
\begin{subequations} \label{Superzlip1}
\begin{eqnarray}
\| F(v) \|_{\mathcal{F}_{p,T}(\Omega)} & \leq & 
C \left(T^{1/p}\|\sigma(0)\|_{\WWW^{1,p}(\Omega)}+
C_R(T)TR\right)  \label{Zest1},\\
\| G(v,\mathfrak{p}) \|_{\mathcal{G}_{p,T}(\Gamma_N)} & \leq & 
C\|\sigma(0)\|_{\WWW^{1,p}(\Omega)} +
CT^{(p+1)/2p^2}
\left(
 CC_R(T)TR + R^2 + R^d 
\right),   \label{Zest2}\\
\|H(v)\|_{\mathcal{H}_{p,T}} & \leq & 
CT^{1/2p^2}\left(R^2 + R^{d} \right),\label{Zest3}
\end{eqnarray}
\end{subequations}
where $C_{R}(T)$ appears in Lemma~\ref{lemma-sigmax}.
Moreover, if $(v_1,\mathfrak{p}_1)$, $(v_2,\mathfrak{p}_2) \in \mathcal{B}_R(T)$, then we have
\begin{subequations} \label{Superzlip2}
\begin{eqnarray}
\| F(v_1)-F(v_2) \|_{\mathcal{F}_{p,T}(\Omega)} & \leq & C C_R(T)T
\| v_1-v_2\|_{\dot{\mathcal{U}}_{p,T}(\Omega)} , \label{Zest11}\\
\| G(v_1,\mathfrak{p}_1) - G(v_2,\mathfrak{p}_2) \|_{\mathcal{G}_{p,T}(\Gamma_N)} & \leq & 
CT^{(p+1)/2p^2}\left(
C_R(T)T\|v_1-v_2\|_{\dot{\mathcal{U}}_{p,T}(\Omega)}\right.
 \nonumber\\
& &  +
\left.(R+R^{d-1})\left(
\|v_1-v_2\|_{\dot{\mathcal{U}}_{p,T}(\Omega)} +
\|p_1-p_2\|_{\mathcal{P}_{p,T}}\right)
\right),\label{Zest12}\\
\|H(v_1)-H(v_2)\|_{\mathcal{H}_{p,T}} & \leq & 
CT^{1/2p^2}(R+R^{d-1}) 
\left(\|v_1-v_2\|_{\dot{\mathcal{U}}_{p,T}(\Omega)}
+ \|v_1-v_2\|_{\L^{\infty}(0,T;\LL^p(\Gamma_N)}
\right). \nonumber \\ & & \label{Zest13}
\end{eqnarray}
\end{subequations}
\end{proposition}

\begin{proof}
Recall~\eqref{NL0}, where we denote $u(\cdot,t) = \displaystyle \int_0^t v(\cdot,s)\d t$, assuming $u_0 = 0$. We have
\begin{equation*}
\| F(v) \|_{\mathcal{F}_{p,T}(\Omega)} = 
\|\divg(\sigma(\nabla u)) \|_{\L^p(0,T;\LL^p(\Omega))} \leq
C\|\sigma(\nabla u) \|_{\L^p(0,T;\WWW^{1,p}(\Omega))}
\end{equation*}
and then~\eqref{Zest1} follows from~\eqref{estlip-sigma0}. Similarly, estimate~\eqref{Zest11} follows from~\eqref{estlip-sigma}. Next, we derive the following estimate, which is non-sharp, but sufficient for our purpose
\begin{eqnarray}
\| G(v,\mathfrak{p}) \|_{\mathcal{G}_{p,T}(\Gamma_N)} & \leq &
\|\sigma(\nabla u)n \|_{\mathcal{G}_{p,T}(\Gamma_N)} +
\left\| \mathfrak{p}\big(\I-\cof(\Phi(u))\big) \right\|_{\mathcal{G}_{p,T}(\Gamma_N)}
 \nonumber \\
& \leq & C\left(\|\sigma(\nabla u)n \|_{\W^{1/(2p'),p}(0,T;\WW^{1-1/p,p}(\Gamma_N))} +
\left\| \mathfrak{p}\big(\I-\cof(\Phi(u))\big)n \right\|_{\W^{1/(2p'),p}(0,T;\WW^{1-1/p,p}(\Gamma_N))}\right.
\nonumber \\
& & \left.+\|\sigma(\nabla u)n \|_{\L^{\infty}(0,T;\LL^{p}(\Gamma_N))} +
\left\| \mathfrak{p}\big(\I-\cof(\Phi(u))\big)n \right\|_{\L^{\infty}(0,T;\LL^{p}(\Gamma_N))}\right)
\nonumber \\
& \leq & C\left(\|\sigma(\nabla u) \|_{\W^{1/(2p'),p}(0,T;\WWW^{1,p}(\Omega))} +
\left\| \mathfrak{p}\big(\I-\cof(\Phi(u))\big) \right\|_{\W^{1/(2p'),p}(0,T;\WWW^{1,p}(\Omega))}\right. \nonumber \\
& & \left. + \|\sigma(\nabla u) \|_{\L^{\infty}(0,T;\WWW^{1,p}(\Omega))} +
\left\| \mathfrak{p}\|_{\L^{\infty}(0,T;\R)}\|\big(\I-\cof(\Phi(u))\big) \right\|_{\L^{\infty}(0,T;\WWW^{1,p}(\Omega))}\right).\label{estk9}
\end{eqnarray}
Using~\eqref{basicest3} with $\alpha = 1/(2p')$, that is $(1-\alpha)/p = (p+1)/2p^2$, we obtain
\begin{equation*}
\|\sigma(\nabla u) \|_{\W^{1/(2p'),p}(0,T;\WW^{1,p}(\Omega))} 
\leq CT^{(p+1)/2p^2}
\|\sigma(\nabla u) \|_{\W^{1,p}(0,T;\WW^{1,p}(\Omega))}.
\end{equation*}
Note that $t\mapsto \I-\cof(\Phi(u))$ vanishes at $t=0$, and from Lemma~\ref{lemma-Brezis} with $\beta = 1/(2p') = \alpha$, we have
\begin{equation*}
\left\| \mathfrak{p}\big(\I-\cof(\Phi(u))\big) \right\|_{\W^{1/(2p'),p}(0,T;\WW^{1,p}(\Omega))}
\leq
CT^{(p+1)/2p^2}\|\mathfrak{p}\|_{\mathcal{P}_{p,T}}
\| \I-\cof(\Phi(u)) \|_{\W^{1,p}(0,T; \WWW^{1,p}(\Omega))},
\end{equation*}
From~\eqref{basicest1} we estimate
\begin{equation*}
\begin{array} {l}
\| \sigma(\nabla u) \|_{\L^{\infty}(0,T; \WWW^{1,p}(\Omega))} \leq 
\| \sigma(0) \|_{\WWW^{1,p}(\Omega)} + 
T^{1/p'} \|\sigma(\nabla u) \|_{\W^{1,p}(0,T;\WWW^{1,p}(\Omega))}, \\
\left\| \I-\cof(\Phi(u)) \right\|_{\W^{1/(2p'),p}(0,T;\WW^{1,p}(\Omega))}
\leq
T^{1/p'}
\| \I-\cof(\Phi(u)) \|_{\W^{1,p}(0,T; \WWW^{1,p}(\Omega))}.
\end{array}
\end{equation*}
Since $(p+1)/2p^2 \leq 1/p'$, we have $T^{1/p'} \leq T^{(p+1)/2p^2}$ when $T\leq 1$. Therefore, from~\eqref{estk9} we deduce
\begin{equation*}
\begin{array} {rcl}
\| G(v,\mathfrak{p}) \|_{\mathcal{G}_{p,T}(\Gamma_N)} & \leq &
C\| \sigma(0) \|_{\WWW^{1,p}(\Omega)} \\
& &  + CT^{(p+1)/2p^2}
\left(
\|\sigma(\nabla u) \|_{\W^{1,p}(0,T;\WW^{1,p}(\Omega))} +
\|\mathfrak{p}\|_{\mathcal{P}_{p,T}}
\| \I-\cof(\Phi(u)) \|_{\W^{1,p}(0,T; \WWW^{1,p}(\Omega))}\right) .
\end{array}
\end{equation*}
By using~\eqref{estlip-sigma0} and~\eqref{est-cof1} we next obtain
\begin{equation*}
\| G(v,\mathfrak{p}) \|_{\mathcal{G}_{p,T}(\Gamma_N)}  \leq  
C\|\sigma(0)\|_{\WWW^{1,p}(\Omega)} +
CT^{(p+1)/2p^2}
\left( CC_R(T)TR + R^2(1+R^{d-2})\right),
\end{equation*}
and thus~\eqref{Zest2}. To prove estimate~\eqref{Zest12}, we write
\begin{equation*}
G(v_1,\mathfrak{p}_1) - G(v_2,\mathfrak{p}_2) =
(\sigma(\nabla u_1)) - \sigma(\nabla u_2))n
+ (\mathfrak{p}_1-\mathfrak{p}_2)\left(\I-\cof (\Phi(u_1))\right)n + \mathfrak{p}_2\left(\cof(\Phi(u_2)) - \cof(\Phi(u_1))\right)n, 
\end{equation*}
and we obtain as previously
\begin{equation*}
\begin{array} {rcl}
\|G(v_1,\mathfrak{p}_1) - G(v_2,\mathfrak{p}_2)\|_{\mathcal{G}_{p,T}(\Gamma_N)} & \leq & CT^{(p+1)/2p^2}\left(
\|\sigma(\nabla u_1)) - \sigma(\nabla u_2)\|_{\W^{1,p}(0,T;\WWW^{1,p}(\Omega))} 
\right. \\
& & 
+\|\mathfrak{p}_1-\mathfrak{p}_2\|_{\mathcal{P}_{p,T}} \|\I-\cof (\Phi(u_1))\|_{\W^{1,p}(0,T;\WWW^{1,p}(\Omega))}
\\
& & \left.
+\|\mathfrak{p}_2\|_{\mathcal{P}_{p,T}} \|\cof (\Phi(u_1)) - \cof(\Phi(u_2))\|_{\W^{1,p}(0,T;\WWW^{1,p}(\Omega))}
\right).
\end{array}
\end{equation*}
We then derive~\eqref{Zest12} by invoking~\eqref{estlip-sigma} and~\eqref{est-cof1}-\eqref{est-cof12}. We estimate the term~\eqref{NL3} as follows
\begin{eqnarray}
|H(v)|_{\R} & \leq &  C\|v\|_{\L^p(\Gamma_N)}\|(\I-\cof(\Phi(u)))n\|_{\LL^{\infty}(\Gamma_N)}
 \leq C\|v\|_{\L^p(\Gamma_N)}\|\I-\cof(\Phi(u))\|_{\WWW^{1,p}(\Omega)}, 
\nonumber \\
\|H(v) \|_{\mathcal{H}_{p,T}} & \leq &
C\left(\left\|
\|v\|_{\LL^p(\Gamma_N))} \|\I-\cof(\Phi(u))\|_{\WWW^{1,p}(\Omega)}
\right\|_{\W^{1-1/2p,p}(0,T;\R)} \right. \nonumber \\
& & \left.+ \|v\|_{\L^{\infty}(0,T;\LL^p(\Gamma_N))}
\|\I-\cof(\Phi(u))\|_{\L^{\infty}(0,T;\WWW^{1,p}(\Omega))}
\right).
\label{est-H}
\end{eqnarray}
Lemma~\ref{lemma-Brezis} with $\beta = 1- 1/(2p)$, that is $(1-\beta)/p = 1/2p^2$, enables us to estimate the first term in the right-hand-side of~\eqref{est-H} as follows
\begin{equation*}
\begin{array} {rcl}
\left\|
\|v\|_{\LL^p(\Gamma_N))} \|\I-\cof(\Phi(u))\|_{\WWW^{1,p}(\Omega)}
\right\|_{\W^{1-1/2p,p}(0,T;\R)} & \leq & 
CT^{1/2p^2} \\
& & \times \left(\| v \|_{\W^{1-1/(2p),p}(0,T;\LL^p(\Gamma_N))}
+ \| v \|_{\L^{\infty}(0,T;\LL^p(\Gamma_N))} \right)\\
& & \times
\| \I-\cof(\Phi(u)) \|_{\W^{1,p}(0,T;\WWW^{1,p}(\Omega))}.  
\end{array}
\end{equation*}
Further, from the trace embedding inequality~\eqref{est-trace-emb} and~\eqref{est-cof1} we deduce
\begin{equation}
\left\|
\|v\|_{\LL^p(\Gamma_N))} \|\I-\cof(\Phi(u))\|_{\WWW^{1,p}(\Omega)}
\right\|_{\W^{1-1/2p,p}(0,T;\R)} \leq 
CT^{1/2p^2}R^2 \left(1+R^{d-2} \right)=
CT^{1/2p^2}(R^2+R^d). \label{est-dsk1}
\end{equation}
We estimate the second term of~\eqref{est-H} by using~\eqref{basicest1} and~\eqref{est-cof1} as follows
\begin{equation}
\|v\|_{\L^{\infty}(0,T;\LL^p(\Gamma_N))}
\|\I-\cof(\Phi(u))\|_{\L^{\infty}(0,T;\WWW^{1,p}(\Omega))} \leq 
CRT^{1/p'} (R+R^{d-1}) = CT^{1/p'}(R^2+R^d).  \label{est-dsk2}
\end{equation}
Since $1/2p^2\leq 1/p'$, we have $T^{1/p'} \leq T^{1/2p^2}$ when $T\leq 1$. Thus, combining~\eqref{est-dsk1} and~\eqref{est-dsk2} in~\eqref{est-H} leads us to~\eqref{Zest3}. In order to obtain~\eqref{Zest13}, we write
\begin{equation*}
H(v_1) - H(v_2) = 
\int_{\Gamma_N} (v_1-v_2)\cdot (\I-\cof(\Phi(u_1)))n\, \d \Gamma_N
+ \int_{\Gamma_N} v_2\cdot \left(\cof(\Phi(u_2)) - \cof(\Phi(u_1))\right)n\, \d\Gamma_N,
\end{equation*}
and proceed as previously, using in particular~\eqref{est-cof12}, in order to get~\eqref{Zest13} and conclude the proof.
\end{proof}

\subsection{Statement of local-in-time wellposedness} \label{sec-mainTh}
We can now prove existence of a unique local-in-time solution to system~\eqref{sysmain2}.

\begin{theorem} \label{th-wellposed} \label{th-locexist}
Under Assumption~$\mathbf{A1}$, there exists $T_0>0$ such that if
\begin{equation*}
f \in \mathcal{F}_{p,T_0}(\Omega), \quad g \in \mathcal{F}_{p,T_0}(\Omega), 
\quad (0, \udotini) \in \mathcal{U}_p^{(0,1)}(\Omega),
\end{equation*}
satisfy the compatibility conditions $\displaystyle \kappa\frac{\p \dot{u}_0}{\p n} + \sigma(0)n = g(\cdot,0)$ on~$\Gamma_N$ and $\displaystyle \int_{\Gamma_N} \dot{u}_0 \cdot n\, \d \Gamma_N = 0$, then system~\eqref{sysmain2} admits a unique solution $(u,\mathfrak{p}) \in \mathcal{U}_{p,T}(\Omega) \times \mathcal{P}_{p,T}$ for all $0 < T\leq T_0$. Further, the following alternative holds:
\begin{itemize}
\item[(i)] Either $T_0 = \infty$,

\item[(ii)] or $\displaystyle \lim_{t \rightarrow T_0}\left(\| (u(t),\dot{u}(t)) \|_{\mathcal{U}_p^{(0,1)}(\Omega)}\right) = \infty$.
\end{itemize}
\end{theorem}

\begin{proof}
Let us show that the mapping $\mathcal{K}$ defined in~\eqref{def-contraction} is a contraction in $\mathcal{B}_R(T)$ (defined in~\eqref{def-ball}) for $R>0$ large enough and $T>0$ small enough, by using the Banach fixed point theorem. Let us first prove the stability of $\mathcal{B}_R(T)$ by $\mathcal{K}$. Let be $(v_a,\mathfrak{p}_a) \in \mathcal{B}_R(T)$, and denote $(v_b,\mathfrak{p}_b) := \mathcal{K}(v_a,\mathfrak{p}_a)$. From Corollary~\ref{coro-linNH}, estimate~\eqref{est-Pruess3} yields
\begin{equation}
\begin{array} {rcl}
\|v_b\|_{\dot{\mathcal{U}}_{p,T}(\Omega)}
+ \|v_b\|_{\L^{\infty}(0,T;\LL^p(\Gamma_N))}
 + \|p_b\|_{\mathcal{P}_{p,T}}
& \leq & C_0(T)\left(
\|\udotini\|_{\WW^{2/{p'},p}(\Omega)} 
+ \|f\|_{\mathcal{F}_{p,T}(\Omega)} 
+ \|g\|_{\mathcal{G}_{p,T}(\Omega)}
\right. \nonumber \\
& & \left. + \|F(v_a)\|_{\mathcal{F}_{p,T}(\Omega)} 
+ \|G(v_a,\mathfrak{p}_a)\|_{\mathcal{G}_{p,T}(\Omega)}  
+ \|H(v_a)\|_{\mathcal{H}_{p,T}} 
\right), 
\end{array}
\label{est-stability}
\end{equation}
where $C_0(T) >0$ is non-decreasing with respect to~$T$, and where the right-hand-sides are defined by~\eqref{NL0}--\eqref{NL3}. Recall that we have first assumed $T\leq 1$. Using the estimates~\eqref{Superzlip1}, we deduce
\begin{equation*}
\begin{array} {l}
 \|v_b\|_{\dot{\mathcal{U}}_{p,T}(\Omega)}
 + \|v_b\|_{\L^{\infty}(0,T;\LL^p(\Gamma_N))}
 + \|\mathfrak{p}_b\|_{\mathcal{P}_{p,T}}
\\
 \leq  C_0(T)\left(
\|\udotini\|_{\WW^{2/{p'},p}(\Omega)} 
+ \|f\|_{\mathcal{F}_{p,T}(\Omega)} 
+ \|g\|_{\mathcal{G}_{p,T}(\Omega)}
+ C\|\sigma(0)\|_{\WWW^{1,p}(\Omega)} + CT^{1/2p^2}(R+R^2+R^d) \right).
\end{array}
\end{equation*}
Now choose $R>0$ large enough, more specifically
\begin{equation*}
R \geq \|\udotini\|_{\WW^{2/{p'},p}(\Omega)} 
+ \|f\|_{\mathcal{F}_{p,T}(\Omega)} 
+ \|g\|_{\mathcal{G}_{p,T}(\Omega)}
+ C\|\sigma(0)\|_{\WWW^{1,p}(\Omega)},
\end{equation*}
and $T>0$ small enough, namely such that $CT^{1/2p^2}(R+R^2+R^d) \leq R$. 
Therefore we obtain
\begin{equation*}
\|v_b\|_{\dot{\mathcal{U}}_{p,T}(\Omega)} 
 + \|v_b\|_{\L^{\infty}(0,T;\LL^p(\Gamma_N))}
+ \|\mathfrak{p}_b\|_{\mathcal{P}_{p,T}}
 \leq  2C_0(T)R,
\end{equation*}
meaning that $(v_b,\mathfrak{p}_b) = \mathcal{K}(v_a,\mathfrak{p}_a) \in \mathcal{B}_R(T)$. Therefore $\mathcal{B}_R(T)$ is stable by $\mathcal{K}$. Further, considering~$(v_a,\mathfrak{p}_a)$ and~$(v_b,\mathfrak{p}_b)$ in~$\mathcal{B}_R(T)$, the difference $(\overline{v},\overline{\mathfrak{p}}) := (v_a-v_b,\mathfrak{p}_a-\mathfrak{p}_b)$ satisfies system~\eqref{sys-Pruess-constraint-NH} with $(F,G,H,v_0)$ replaced by $(F(v_a)-F(v_b), G(v_a,\mathfrak{p}_a)-G(v_b,\mathfrak{p}_b), H(v_a)-H(v_b), 0)$ as data. Therefore it satisfies the estimate~\eqref{est-Pruess3} with the corresponding right-hand-sides, namely
\begin{equation*}
\begin{array} {l}
\|\overline{v}\|_{\dot{\mathcal{U}}_{p,T}(\Omega)} 
 + \|\overline{v}\|_{\L^{\infty}(0,T;\LL^p(\Gamma_N))}
+ \|\overline{\mathfrak{p}}\|_{\mathcal{P}_{p,T}} \\
 \leq C_0(T)\left(
\|F(v_a)-F(v_b)\|_{\mathcal{F}_{p,T}(\Omega)}\| + 
\|G(v_a,\mathfrak{p}_a)-G(v_b,\mathfrak{p}_b)\|_{\mathcal{G}_{p,T}(\Gamma_N)} + \|H(v_a)-H(v_b)\|_{\mathcal{H}_{p,T}} \right).
\end{array}
\end{equation*}
By using the estimates~\eqref{Superzlip2} we obtain
\begin{equation*}
\begin{array} {l}
\|\overline{v}\|_{\dot{\mathcal{U}}_{p,T}(\Omega)} 
 + \|\overline{v}\|_{\L^{\infty}(0,T;\LL^p(\Gamma_N))}
+ \|\overline{\mathfrak{p}}\|_{\mathcal{P}_{p,T}} \\
 \leq C_0(T) CT^{1/2p^2}\left(
\|\overline{v}\|_{\dot{\mathcal{U}}_{p,T}(\Omega)} 
 + \|\overline{v}\|_{\L^{\infty}(0,T;\LL^p(\Gamma_N))}
+ \|\overline{\mathfrak{p}}\|_{\mathcal{P}_{p,T}}
\right),
\end{array}
\end{equation*}
and, again by choosing $T>0$ small enough, we make $\mathcal{K}$ a contraction in~$\mathcal{B}_R(T)$. Thus there exists $T_0>0$ such that for all $T\leq T_0$ system~\eqref{sysmain2} admits unique solution $(v,\mathfrak{p})$. The alternative is obtained classically via a continuation argument: Defining $T_0$ as the maximal time of existence of the solution $(u,\dot{u})$ so obtained, namely
\begin{equation*}
T_0 = \sup \left(\mathcal{T} := \left\{ T >0 \mid  (u,\dot{u}) \in \mathcal{U}_{p,T}(\Omega) \times  \dot{\mathcal{U}}_{p,T}(\Omega) \text{ exists} \right\}\right).
\end{equation*}
We just showed that the set~$\mathcal{T}$ is non-empty. By contradiction, assume that $T_0< \infty$ and that 
\begin{equation*}
\displaystyle \lim_{t \rightarrow T_0}\left(\| (u(t),\dot{u}(t)) \|_{\mathcal{U}_p^{(0,1)}(\Omega)}\right) < \infty.
\end{equation*}
Then $(u(T_0),\dot{u}(T_0)) \in \mathcal{U}_p^{(0,1)}(\Omega)$. From what precedes we can extend the solution to an interval $(T_0, T_0+\eta)$ for some $\eta>0$. This contradicts the definition $T_0$ as an upper bound, and concludes the proof.
\end{proof}

\newpage
\appendix

\section{Modeling aspects} \label{sec-modeling}
In this section we address the modeling aspects of the problem, in particular the global injectivity constraint and the way system~\eqref{sysmain} is derived, as well as the form of the control operator.

\subsection{On the global injectivity condition and the invertibility condition} \label{sec-modeling1}

The so-called global injectivity condition, studied by Ciarlet-Ne\v{c}as~\cite{Necas1987}, writes
\begin{equation}
\int_{\Omega} \left( \det\left(I+\nabla u(\cdot,t) \right) - \det(\I+\nabla u_0) \right)\d\Omega =  0, \quad \text{for all } t\in (0,T). \label{contdet}
\end{equation}
After derivation in time, and by using the Piola identity and the divergence formula, it is equivalent to
\begin{equation*}
\int_{\Omega} \cof(\I+\nabla u) : \nabla \dot{u} \, \d \Omega = 
\int_{\Omega} \divg\left(\cof(\I+\nabla u)^T\dot{u} \right) \d \Omega =
\int_{\p \Omega}  \dot{u} \cdot \cof\left(I+\nabla u \right)n\, \d\Omega = 0.
\end{equation*}
Since we assume that $\dot{u} = 0 $ on $\Gamma_D$, we then consider the following equivalent constraint:
\begin{equation}
\int_{\Gamma_N}  \dot{u} \cdot \cof\left(I+\nabla u \right)n\, \d\Omega = 0.
\label{eq-constraint}
\end{equation}
Furthermore, for the sake of consistency, modeling elastic deformations requires to guarantee that the mapping $\Id +u(\cdot,t)$ is invertible for $t \geq 0$. Actually, assuming that $\Id +u(\cdot,0)$ is invertible, and under regularity assumptions, this invertibility condition can be relaxed, provided that~$t>0$ is small enough. More precisely, we have the following result:

\begin{lemma}
There exists a constant $C>0$ such that for all $u\in \W^{1,p}(0,T; \WW^{1,p}(\Omega))$ the following estimate
\begin{equation*}
\| \det(\I+\nabla u(\cdot,t)) - \det(\I + \nabla u(\cdot,0))\|_{\L^1(\Omega)}
\leq  t^{1-1/p} C\left(1+ \| \nabla u \|^{d-1}_{\L^{\infty}(0,T;\LL^p(\Omega))} \right) 
\| \nabla \dot{u} \|_{\L^p(0,T;\LL^p(\Omega))}
\end{equation*}
 holds for all $t \in [0,T]$.
\end{lemma}

\begin{proof}
The result is provided by~\cite[Lemma~3]{Court2017}.
\end{proof}

\noindent Therefore, if $\det(\I+\nabla u(\cdot,0)) >0$ a.e. in $\Omega$, by choosing~$T>0$ small enough
we deduce that $\det(\I+\nabla u(\cdot,t)) >0$ also, and $\Id +u(\cdot,t)$ remains invertible.


\subsection{Derivation of the PDE system} \label{sec-derivPDE}
Let us explain how the system~\eqref{sysmain} of partial differential equations can be derived from the least action principle. The kinetic energy of the system and the potential stored energy are respectively given by
\begin{equation*}
	\frac{1}{2}\int_{\Omega} \rho | \dot{u} |^2 \d \Omega, 
	\quad \text{ and } \quad 	\int_{\Omega} \mathcal{W}(E(u)) \d \Omega,
\end{equation*}
where $\rho>0$ denotes the density of the material, and $E$ denotes the so-called Green -- St-Venant strain tensor. 
We consider for hyperelastic materials some general strain energy function $\mathcal{W}(E)$ satisfying assumptions~$\mathbf{A1}-\mathbf{A2}$. Recall the notation $\Phi(u) = \I+ \nabla u$., and note that
\begin{equation*}
\begin{array} {rcl}
\displaystyle
	\frac{\p \mathcal{W}(E(u))}{\p u}.v  = 
	\frac{\p \mathcal{W}}{\p E}(E(u)): (E'(u).v) =
\displaystyle
\frac{1}{2}\check{\Sigma}(E(u)): \left(
	\Phi(u)^T\nabla v + \nabla v^T\Phi(u) \right) =
	\Sigma(u):\left(\Phi(u)^T \nabla v\right),
\end{array}
\end{equation*}
as the tensor $\check{\Sigma}(E)$ is assumed to be symmetric.
Denoting by $\mathfrak{p}$ a Lagrange multiplier for the constraint~\eqref{contdet}, we consider a saddle-point of the following Lagrangian functional:
\begin{equation*}
\begin{array}{rcl}
L(u,\dot{u},\mathfrak{p}) & = & \displaystyle
\left(\frac{1}{2}\int_{\Omega} \left(\rho | \dot{u} |^2 \d \Omega
- \mathcal{W}(E(u)) \right) \d \Omega 
+ \int_{\Omega} f \cdot u\, \d \Omega 
+ \int_{\Gamma_N} g \cdot u\, \d \Gamma_N
\right)  \\[10pt]
& & \displaystyle
-  \mathfrak{p}\int_{\Omega} \left( \det\left(\Phi(u) \right) - \det\left(\Phi(u_0) \right) \right)\d\Omega  .
\end{array}
\end{equation*}
Using the Green formula, the first-order derivatives of $L$ are obtained as follows
\begin{equation*}
\begin{array} {rcl}
\displaystyle	\frac{\delta L}{\delta u} & = & \displaystyle
\frac{\p L}{\p u} - \frac{\d}{\d t}\left(\frac{\p L}{\p \dot{u}}\right),\\
\displaystyle\frac{\delta L}{\delta u}(u,\mathfrak{p}).v & = & \displaystyle
-\int_{\Omega} \rho \ddot{u} \cdot v \, \d \Omega -
\int_{\Omega} \nabla v : (\Phi(u)\Sigma(u)) \d \Omega 
+ \int_{\Omega} f \cdot v\, \d \Omega  
+ \int_{\Omega} g \cdot v\, \d \Gamma_N\\
& & \displaystyle
- \mathfrak{p}\int_{\Omega} \cof\left(\Phi(u) \right) : \nabla v\, \d\Omega \\[10pt]
& = & \displaystyle
-\int_{\Omega} \left(\rho \ddot{u} - \divg (\Phi(u)\Sigma(u)) - f\right) \cdot v \, \d \Omega \\[10pt]
& & \displaystyle- \int_{\Gamma_N} \left( \Phi(u)\Sigma(u)n  
+ \mathfrak{p}\, \cof\left(\Phi(u) \right)n -g \right)\cdot v \, \d\Gamma_N,\\[10pt]
\displaystyle\frac{\delta L}{\delta \mathfrak{p}}(u,\mathfrak{p}) & = & \displaystyle
-\int_{\Omega} \left( \det\left(\Phi(u) \right) - \det\left(\Phi(u_0) \right) \right)\d\Omega.
\end{array}
\end{equation*}
Thus, from the Euler-Lagrange equation, a critical point $(u,\mathfrak{p})$ of the functional $L$ satisfies the Ciarlet-Ne\v{c}as condition~\eqref{contdet}, and
\begin{equation*}
\begin{array} {rcl}
\rho\ddot{u}
- \divg \sigma(\nabla u) = f & & \text{in $\Omega \times(0,T)$}, \label{probnonlin1} \\
\sigma(\nabla u)n  + p\ \cof\left(F(u)\right) n = g & & \text{on $\Gamma_N\times(0,T)$},
\end{array}
\end{equation*}
recalling the notation $\sigma(\nabla u) = \Phi(u) \Sigma(u)$. For the sake of simplicity, and without loss of generality, we choose $\rho \equiv 1$. The first equation above is hyperbolic. For mathematical purpose we introduce a parabolic regularization, by adding the diffusion term ~$-\kappa\Delta \dot{u}$ in the first equation, and its corresponding Neumann term~$\kappa\displaystyle \frac{\p \dot{u}}{\p n}$ in the second equation, for some constant~$\kappa>0$. Replacing equivalently the Ciarlet-Ne\v{c}as condition by its time-derivative~\eqref{eq-constraint}, the resulting system is system~\eqref{sysmain2}, equivalent to system~\eqref{sysmain}.

\subsection{The pressure as a function of the displacement field} \label{sec-pressure}
Keep in mind that the pressure does not depend on the space variable. In the case where~$g = 0$, One can multiply the second equation of~\eqref{sysmain} by $\cof(\I+\nabla u)^{-1} = \det(\I+\nabla u)^{-1} (\I+\nabla u)^{T} = \det(\Phi(u))^{-1} \Phi(u)^T$, and then one derives an expression for the pressure, written in terms of $(u,\dot{u})$, as follows:
\begin{equation}
\mathfrak{p} = -\frac{1}{\Gamma_N} \int_{\Gamma_N} \det(\Phi(u))^{-1} \Phi(u)^T 
\left( \kappa \frac{\p \dot{u}}{\p n} + \sigma(\nabla u) n - g \right)\cdot n\, \d \Gamma_N.
\label{exp-pressure}
\end{equation}

\begin{remark}
\noindent Note that in the formal case where $\kappa = 0$ and $g=0$, and for strain energy density functions that are function of the so-called symmetric {\it right Cauchy–Green deformation tensor} $C:= \Phi(u)^T\Phi(u)$, so that we have $\sigma(\nabla u) = \Phi(u)\hat{\Sigma}(C)$, and the pressure $\mathfrak{p}$ can be expressed only in terms of $C$ as follows
\begin{equation*}
	\mathfrak{p}\, n  =  - \det(C)^{-1/2} C\hat{\Sigma}(C)n, \quad
	\mathfrak{p}  =  - \frac{1}{\Gamma_N}\int_{\Gamma_N}
	\det(C)^{-1/2} \hat{\Sigma}(C)n \cdot Cn \, \d \Gamma_N.
\end{equation*}
Thus, physically, $\mathfrak{p}$ is a function of local change in distances due to the deformation $x\mapsto \Id+u$.
\end{remark}
Using~\eqref{exp-pressure} and the same previous techniques of estimation, we can deduce the following regularity for the pressure:
\begin{equation*}
\left\{ \begin{array} {l}
u\in \mathcal{U}_{p,T}(\Omega) \\
g \in \mathcal{G}_{p,T}(\Gamma_N)
\end{array}\right.
\Rightarrow
\mathfrak{p} \in \W^{1/2p',p}(0,T;\R) = \mathcal{P}_{p,T}.
\end{equation*}
In particular, the pressure~$\mathfrak{p}$ is continuous in time.

\begin{remark}
Assume~$\kappa=0$ and $g=0$. Taking the scalar product of the second equation of~\eqref{sysmain} by any test function~$v$, and integrating over~$\Gamma_N$, we get
\begin{equation*}
\int_{\Gamma_N} v\cdot \sigma(\nabla u)n\, \d \Gamma_N +
\mathfrak{p}\int_{\Gamma_N} v \cdot \cof(\Phi(u))n\, \d \Gamma_N = 0
\end{equation*}
The displacement~$u$ defines the deformation~$\Id + u$. Recall that~$\sigma(\nabla u) = \Phi(u)\Sigma(u)$, where~$\Sigma$ is also called the second Piola-Kirchhoff stress tensor. It is related with the so-called Cauchy stress tensor~$\mathcal{T}$ -- defined in the deformed configuration~$(\Id+u)(\Gamma_N)$ -- via the following relation
\begin{equation*}
\sigma(\nabla u) = \Phi(u) \Sigma(u) = (\mathcal{T} \circ (\Id + u))\, \cof(\Phi(u)).
\end{equation*}
The integrals on~$\Gamma_N$ are transformed into integrals on the deformed boundary~$\Gamma_N(t):=(\Id+u)(\Gamma_N)$ as follows
\begin{equation*}
\int_{\Gamma(t)} (v\circ(\Id+u)^{-1}) \cdot \mathcal{T}n\, \d \Gamma_N(t) +
\mathfrak{p} \int_{\Gamma_N(t)} (v\circ(\Id+u)^{-1}) \cdot n\, \d \Gamma_N(t).
\end{equation*}
See~\cite[formula~$(14)$, page~51]{Gurtin}. Now choosing~$v$ such that~$(v\circ(\Id+u)^{-1}) = n$ on~$\Gamma_N(t)$, namely~$v= (\cof(\Phi(u))n)/|\cof(\Phi(u))n|_{\R^d}$ on the reference configuration~$\Gamma_N$, we obtain
\begin{equation*}
\mathfrak{p} = -\frac{1}{|\Gamma_N(t)|} \int_{\Gamma_N(t)} 
\mathcal{T}n \cdot n\, \d \Gamma_N(t),
\end{equation*}
showing that the pressure writes simply in terms of the Cauchy stress tensor, more specifically via a traction term on the deformed boundary~$\Gamma_N(t)=(\Id +u)(\Gamma_N)$.
\end{remark}

\section{Proof of Lemma~\ref{lemma-DiNezza}} \label{sec-app-B}
Let us show that functions of~$\W^{\gamma,p}(0,T;B)$ are $(\gamma-1/p)$-H\"{o}lder. For any $r>0$ and $t\in (0,T)$, introduce
\begin{equation*}
\omega_{t,r} := \{s\in (0,T) \mid |s-t| \leq r \} \cap (0,T), 
\quad \quad \quad 
\langle \varphi \rangle_{t,r} := 
\displaystyle \frac{1}{|\omega_{t,r}|}\int_{\omega_{t,r}}\varphi(s) \d s.
\end{equation*}
We write
\begin{equation}
| \varphi(t) -\varphi(0) | \leq
| \varphi(t) - \langle \varphi \rangle_{t,t} | + 
| \langle \varphi \rangle_{t,t} - \langle \varphi \rangle_{0,t} | 
+ | \langle \varphi \rangle_{0,t} -\varphi(0) |.
\label{ineq-triangle}
\end{equation}
Note that $|\omega_{0,t}| = t$ and $t \leq |\omega_{t,t}| \leq 2t$. We first estimate the second term of the right-hand-side, by using the H\"{o}lder's inequality, as follows
\begin{equation*}
\begin{array}{rcl}
| \langle \varphi \rangle_{t,t} - \langle \varphi \rangle_{0,t} | 
& \leq & \displaystyle \frac{1}{t|\omega_{t,t}|}\int_0^t \int_{\omega_{t,t}}
|\varphi(s)-\varphi(s') | \d s \d s' \\
& \leq & \displaystyle \frac{1}{(t|\omega_{t,t}|)^{1/p}}
\left(\int_0^t \int_{\omega_{t,t}}
|\varphi(s)-\varphi(s') |^p \d s \d s' \right)^{1/p}\\
& \leq & \displaystyle \frac{1}{(t|\omega_{t,t}|)^{1/p}}
\left(\sup_{\text{\begin{small}$s\in \omega_{t,t},\ s' \in \omega_{0,t}$\end{small}}}
|s-s'|^{\gamma+1/p} \right)\| \varphi\|_{\W^{\gamma,p}(0,T;B)} \\
& \leq & \displaystyle \frac{1}{t^{2/p}}(2t)^{\gamma +1/p} \| \varphi\|_{\W^{\gamma,p}(0,T;B)},
\end{array}
\end{equation*}
referring to~\eqref{def-fractional} for the definition of~$\|\varphi\|_{\W^{\gamma,p}(0,T;B)}$. Thus
\begin{equation}
| \langle \varphi \rangle_{t,t} - \langle \varphi \rangle_{0,t} | \leq
Ct^{\gamma-1/p} \|\varphi\|_{\W^{\gamma,p}(0,T;B)}.
\label{zzz1}
\end{equation}
The first and third terms of~\eqref{ineq-triangle} are treated similarly. The Lebesgue differentiation theorem states that $
\displaystyle \lim_{r\rightarrow 0} \langle \varphi \rangle_{\tau,r} = \tau$ for any $0 \leq \tau \leq T$, that we use for $\tau \in \{0,t\}$. For any $r>0$, we estimate as previously
\begin{equation*}
\begin{array} {rcl}
| \langle \varphi \rangle_{\tau,r} - \langle \varphi \rangle_{\tau,2r}| & \leq &
\displaystyle \frac{1}{|\omega_{\tau,r}| |\omega_{\tau,2r}|}
\int_{\omega_{\tau,r}}\int_{\omega_{\tau,2r}} |\varphi(s) -\varphi(s')|\d s \d s' \\
& \leq &
\displaystyle \frac{1}{(|\omega_{\tau,r}| |\omega_{\tau,2r}|)^{1/p}}
\left(\sup_{\text{\begin{small}$s\in \omega_{\tau,2r},\ s' \in \omega_{\tau,r}$\end{small}}}
|s-s'|^{\gamma+1/p} \right)
\| \varphi \|_{\W^{\gamma,p}(0,T;B)} \\
& \leq & C |\omega_{\tau,r}|^{-1/p} |\omega_{\tau,2r}|^{\gamma}
\| \varphi \|_{\W^{\gamma,p}(0,T;B)} \\
& \leq & C r^{-1/p} (4r)^{\gamma}
\| \varphi \|_{\W^{\gamma,p}(0,T;B)} \leq 
Cr^{\gamma-1/p}
\| \varphi \|_{\W^{\gamma,p}(0,T;B)},
\end{array}
\end{equation*}
where we used $|\omega_{\tau,r}|\geq r$ and $|\omega_{\tau,2r}| \leq 4r$. Now we choose $r = r_i := 2^{-i}t$, and for any $k\in \mathbb{N}$ we deduce
\begin{equation*}
\begin{array} {rcl}
| \langle \varphi \rangle_{\tau,r_k} - \langle \varphi \rangle_{\tau,t}| & \leq &
\displaystyle \sum_{i=1}^k 
| \langle \varphi \rangle_{\tau,r_i} - \langle \varphi \rangle_{\tau,r_{i-1}}|
\leq
\displaystyle \sum_{i=1}^k 
| \langle \varphi \rangle_{\tau,r_i} - \langle \varphi \rangle_{\tau,2r_i}| 
 \\
& \leq & Ct^{\gamma-1/p} \|\varphi\|_{\W^{\gamma,p}(0,T;B)}
\left(\displaystyle \sum_{i=1}^k 2^{-i(\gamma-1/p)} \right)
\leq  Ct^{\gamma-1/p} \|\varphi\|_{\W^{\gamma,p}(0,T;B)}.
\end{array}
\end{equation*}
Passing to the limit when $k\rightarrow \infty$, we obtain
\begin{equation*}
|  \varphi (\tau) - \langle \varphi \rangle_{\tau,t}|
\leq Ct^{\gamma-1/p} \|\varphi\|_{\W^{\gamma,p}(0,T;B)}.
\end{equation*}
By choosing $\tau=0$ and $\tau=t$ we then get
\begin{equation}
|  \varphi (0) - \langle \varphi \rangle_{0,t}| + 
|  \varphi (t) - \langle \varphi \rangle_{t,t}|
\leq Ct^{\gamma-1/p} \|\varphi\|_{\W^{\gamma,p}(0,T;B)}.
\label{zzz2}
\end{equation}
Combining~\eqref{ineq-triangle}, \eqref{zzz1} and \eqref{zzz2} yields the first announced estimate. The second estimate follows by the triangular inequality, and completes the proof.

\setstretch{1}
\section*{Statement}
The author confirms that he is the single contributor to this work, that the latter did not receive any funding, and does not present any conflict of interest. No data were needed for achieving this work.

\printbibliography

\end{document}